\documentclass[11pt, a4paper]{article}
\usepackage{t1enc}
\usepackage[latin1]{inputenc}
\usepackage[english]{babel}
\usepackage{amsmath,amsthm}
\usepackage{amsfonts}
\usepackage{latexsym}
\usepackage[dvips]{graphicx}
\usepackage{float}
\usepackage[natural]{xcolor}
\usepackage{algorithm}
\usepackage{algorithmic}
\usepackage{enumerate,enumitem}
\usepackage{multirow}
\usepackage{xcolor}
\usepackage[colorlinks,linkcolor=blue]{hyperref}
\usepackage{lineno}
\usepackage{setspace}
\usepackage{fullpage}
\usepackage[title]{appendix}
\usepackage{todonotes}
\usepackage{amssymb}

\usepackage{listings}
\usepackage{bbm}

\usepackage{fullpage}
\usepackage{hyperref}
\usepackage[margin=2.3cm]{geometry}
\usepackage{enumitem,verbatim}

\newtheorem{theorem}{Theorem}[section]
\newtheorem{claim}[theorem]{Claim}
\newtheorem{subclaim}[theorem]{Subclaim}
\newtheorem{lemma}[theorem]{Lemma}
\newtheorem{conj}[theorem]{Conjecture}

\newcounter{propcounter}

\theoremstyle{definition}

\newtheorem{definition}[theorem]{Definition}
\newcommand{\eps}{\varepsilon}
\newcommand{\<}{\subseteq}

\newcommand{\ft}[1]{\footnote{#1}}

\title{\LARGE Embedding clique subdivisions via crux}

\author{
Donglei Yang\thanks{School of Mathematics, Shandong University, Jinan, Shandong, China. Email: {\tt dlyang@sdu.edu.cn}. Supported by Natural Science Foundation of China (12101365) and by Natural Science Foundation of Shandong Province (ZR2021QA029).}
\and
Fan Yang \thanks{Data Science Institute, Shandong University, Jinan, Shandong, China. Email: {\tt fyang@sdu.edu.cn}. Supported by Natural Science Foundation of China (12301447) and by China Postdoctoral Science Foundation (12570073310023).}
\and
}
\date{}

\begin{document}
\maketitle
\begin{abstract}
For a graph $G$ and a constant $\alpha>0$, we denote by $C_{\alpha}(G)$ the minimum order of a subgraph $H\<G$ with $d(H)\ge \alpha d(G)$. Liu and Montgomery conjectured that every graph $G$ contains $K_{\Omega(t)}$ as a subdivision for $t=\min \{d(G), \sqrt{\tfrac{C_{\alpha}(G)}{\log C_{\alpha}(G)}}\}$. In the paper, we prove this conjecture.
\end{abstract}

\section{Introduction}
For a graph $H$, a \emph{subdivision} of $H$, denoted by $TH$, is a graph obtained by replacing edges of $H$ by internally vertex-disjoint paths. This is a fundamental concept for studying topological aspects of graphs since Kuratowski \cite{Kura} (in 1930) used this notion to characterize the planar graphs, proving that a graph is planar if and only if it does not contain a subdivision of $K_5$ or $K_{3,3}$. A well-studied direction of research in extremal graph theory is to find sufficient conditions on a graph $G$ that would guarantee the existence of an $H$-subdivision in $G$.

Such problem has received considerable attention since the work of Mader \cite{ma67} from 1967 on the optimal average degree conditions forcing subdivisions of large cliques. More precisely, Mader showed that for every $k\in \mathbb{N}$, there exists (finite) $f(k)$ such that every graph $G$ with average degree at least $f(k)$ contains a $TK_k$. Mader \cite{ma67}, and independently Erd\H{o}s and Hajnal \cite{Erd}, conjectured that one can take $f(k)=O(k^2)$. This conjecture was finally resolved in 1999 by Bollob\'{a}s and Thomason \cite{BolTto} and independently by Koml\'{o}s and Szemer\'{e}di \cite{KS}. In fact, Jung \cite{Jung} observed that $K_{k^2/8,k^2/8}$ does not contain $TK_k$, which implies that the quadratic bound on $f(k)$ is best possible.

Note that the extremal examples mentioned above consist of (disjoint union of) complete bipartite graphs. The following question naturally arises: can we find a larger clique subdivision in $G$ if it does not structurally look like (disjoint union of) dense bipartite graph? In 1999, Mader \cite{C4free} conjectured that any $C_4$-free graph contains a subdivision of a clique with order linear in its average degree. This linear bound is essentially optimal under the obvious ``degree constraint'' as the host graph could be regular. After some partial results (see e.g.~\cite{C6free,largegirth1,largegirth2}), Mader's conjecture was subsequently resolved by Liu and Montgomery \cite{LiuC4} in 2017. They also proved an analogous result for $K_{s,t}$-free graphs, asserting that for $t\geq s\geq 2$, every $K_{s,t}$-free graph $G$ with $d(G)=d$ contains a $TK_{\Omega(d^{s/2(s-1)})}$.

The exclusion of $K_{s,t}$ has proven to be a good way to formalize the stability-type question aforementioned. Liu and Montgomery \cite{LiuC4} also suggested another way to formalize this: whether a graph with no small subgraphs which are almost as dense as the parent graph must contain a large clique subdivision? %generalising this to non-complete bipartite forbidden subgraphs.
To formulate this, they introduced a new graph parameter depicting the order of the smallest dense patch of a
graph, which was recently named \emph{crux} by Haslegrave Hu, Kim, Liu, Luan and Wang \cite{Hasle}. The notion of crux has recently been used to capture the ``essential order'' of a graph with applications to finding long cycles in graphs and random subgraphs (see e.g.~\cite{Hasle}), though similar idea already appeared implicitly in previous works due to for example Krivelevich and Samotij \cite{KS-cycle}.
\begin{definition}[Crux]
  Given a graph $G$ and a constant $\alpha >0$, a subgraph $H\<G$ is an \emph{$\alpha$-crux} of $G$ if $d(H)\ge \alpha d(G)$. Let $C_{\alpha}(G)$ be the minimum order over all $\alpha$-cruxs of $G$.
\end{definition}

%For instance, condition on chromatic number was proposed by Haj\'{o}s, who conjectured in 1961 a strengthening of Hadwiger's conjecture that every graph $G$ with chromatic number $\chi(G)\geq t$ contains a $TK_t$. Dirac \cite{sub4} showed that this conjecture is true for $t\leq 4$, but in 1979 Catlin \cite{Hajos} disproved the conjecture for all $t\geq 7$. Later, Erd\H{o}s and Fajtlowicz \cite{Erd} showed that the conjecture is false for almost all graphs by considering random graphs, see also~\cite{largegirth1,largegirth2} for more recent developments.
%As a stronger and more fundamental question, conditions on average degree guaranteeing an $H$-subdivision have been extensively studied, starting from a result of Mader . the following stability-type question naturally arises. Can we find a larger clique subdivision in $G$ if it does not structurally look like (disjoint union of) dense bipartite graph?
%One way to formalize this question was suggested by Mader \cite{C4free}, conjecturing that the quadratic bound can be improved to a linear one; that is, every $C_4$-free graph $G$ with average degree $d(G)=\Omega(k)$ contains a $TK_{k}$.
%A question was suggested by Liu and Montgomery \cite{LiuC4}.

\noindent Observing that the disjoint union of dense bipartite graphs have $C_{\alpha}(G)$ linear in the average degree, Liu and Montgomery \cite{LiuC4} raised the following conjecture by taking $\alpha=\tfrac{1}{100}$.

\begin{conj}[\cite{LiuC4}]\label{conj}
There exists some constant $c>0$ such that every graph $G$ contains a subdivision of a clique with at least
$c\min \Big\{d(G), \sqrt{\tfrac{C_{\alpha}(G)}{\log C_{\alpha}(G)}}\Big\}$ vertices.
\end{conj}
\noindent If true, this bound is tight: as noted in \cite{LiuC4}, the $d/1000$-blow-up $G$ of a $d$-vertex $1000$-regular expander satisfies $C_{\alpha}(G)=\Theta(d^2)$ and the largest clique subdivision has order $O(d/\sqrt{\log d})$. For convenience, we will often call this construction a \textit{space barrier}. To see this, note that $G$ has $\Omega(d^2)$ vertices and any $K_t$-subdivision in $G$ has $\tbinom{t}{2}$ internally vertex-disjoint paths. As discussed in \cite{LiuC4}, when $t$ is close to $d$, among any collection of $t$ vertices in $G$, there are at least $\tfrac{1}{4}t^2$ pairs of vertices that are far apart from each other by a distance $\Omega(\log d)$. This tells that at least half of the subdivided paths must contain at least $\Omega(\log d)$ vertices, which implies $t=O(d/\sqrt{\log d})$.

Note that Im, Kim, Kim and Liu \cite{crux} recently proved a weaker bound for $\alpha=\tfrac{1}{100}$.
\begin{theorem}[\cite{crux}]\label{im}
There exists an absolute constant $\beta>0$ such that the following is true.
Let $G$ be a graph with $d(G)=d$. Then $G$ contains a $K_{\beta t/(\log\log t)^6}$-subdivision where
$$t=  \min\Big\{d, \sqrt{\tfrac{C_\alpha(G)}{\log C_\alpha(G) }}\Big\}.$$
\end{theorem}
\noindent
Our main result answers Conjectur~\ref{conj} in the affirmative.

\begin{theorem}\label{main1}
There exists an absolute constant $c>0$ such that the following is true for any $0<\alpha\le\tfrac{1}{9\times 10^4}$. Let $G$ be a graph with $d(G)=d$. Then $G$ contains a $K_{ct}$-subdivision where
\[t=\min\Big\{d,\sqrt{\tfrac{C_\alpha(G)}{\log{C_\alpha(G)}}}\Big\}.
\]
\end{theorem}
We make no attempt to optimize the constant coefficients $c$ and $\alpha$ in Theorem \ref{main1}. \medskip

\noindent\textbf{Technical Contribution.} Our approach builds on the recently-developed sublinear expander theory and introduces new perspectives on the $\alpha$-crux number $C_{\alpha}(G)$. %The theory of sublinear expander has found a wealth of applications in the recent resolutions of several long-standing conjectures.
To build large clique subdivisions, we adopt a strategy of Liu and Montgomery \cite{LiuC4} (also appeared in previous works \cite{montgo}), namely, to find a sufficient number of star-like structures called ``units'' that serve as the bases for constructing the required subdivision. To do this, it is desired that all these units do not overlap too much and each has a large exterior and (relatively) small interior. Then the arguments often reduce to robustly connecting the exteriors with disjoint short paths whilst avoiding the interiors, using the connectivity provided by the sublinear expander. In this paper, we use a variant of unit structure called \textit{web} (see Definition \ref{defn:web}) and the bulk of the work is therefore to construct a family of webs as above within a sublinear expander. Note that the extremal construction (space barrier) tells that the (almost half) subdivided paths in the $K_{\Omega(t)}$-subdivision has length $\Theta(\log n)$. This requires us to be able to find disjoint short connections of length $O(\log n)$ , which is in fact a surprisingly challenging task in previous applications of sublinear expander (e.g. in \cite{crux}, they were able to find disjoint subdivided paths of length $\log n(\log\log n)^{O(1)}$). To see this, note that in the space barrier, all the subdivided paths have too many (up to $\Omega(n)$) vertices to naively find one more path which avoids previous paths just by applying the commonly-used Lemma~\ref{distance}. This is where our new perspectives on the $\alpha$-crux number and a random variant of Lemma~\ref{distance} come into play (see Section~\ref{sec_3.1} for a more elaborate sketch of the proof)..

The proof of Theorem \ref{main1} is divided into three main cases by distinguishing the value of the $\alpha$-crux number $C_{\alpha}(G)$. We employ an unified construction by
combining our new perspectives on the $\alpha$-crux number with an adaptation of random partitions for sublinear expanders appearing in a very recent work of Buci\'{c} and Montgomery \cite{Mon-cycle}.
%which combines the (strong) expansion property from the crux with the (robust) connectivity of the sublinear expander

%

\section{Preliminaries}\label{Pre2}
\subsection{Notation}
For any positive integer $r$, we write $[r]$ for the set $\{1, \ldots, r\}$.
Given a graph $G=(V,E)$, we denote by $v(G),e(G),\Delta(G),\delta(G)$ the number of vertices in $G$, the size of $G$, the maximum degree of $G$ and the minimum degree of $G$. Given a set $W\subseteq V(G)$, we write $N_G(W)=(\bigcup_{u\in W}N_G(u))\setminus W$. Furthermore, set $N_G^{0}(W):=W$, $N_G^1(W):=N_G(W)$ and $N_G[W]=N_G(W)\cup W$, and for each $i\geq 1$, define $N_G^{i+1}(W):=N(N_G^{i}(W))\setminus N_G^{(i-1)}(W)$. Denote by $B_G^r(W)$ the ball of radius $r$ around $W$, that is, $B_G^r(W)=\bigcup_{i\leq r}N_G^{i}(W)$. For simplicity, write $B_G^r(v)=B_G^r(\{v\})$. For any set $W\subset V(G)$, the subgraph of $G$ induced on $W$, denoted as $G[W]$, is the graph with vertex set $W$ and edge set $\{xy\in E(G)|x,y\in W\}$, and write $G-W=G[V(G)\setminus W]$.
For any $A,B\subseteq V(G)$, we denote by $G[A,B]$ the graph with vertex set $A\cup B$ and edge set $\{xy\in E(G)|x\in A, y\in B\}$. We simply use $e_G(A,B)=|E(G[A,B])|$. %Moreover, we define the \emph{density} between $A$ and $B$ to be
%\begin{equation*}
%d_G(A,B)=\tfrac{e_G(A,B)}{|A||B|}.
%\end{equation*}

For a path $P$, the length of $P$ is the number of edges in $P$, and we say $P$ is an $x,y$-path if $x$ and $y$ are the endvertices of $P$. A star is called \emph{$x$-star} for some $x\in \mathbb{N}$ if it has exactly $x$ leaves. Given a family of graphs $\mathcal{F}$, denote by $|\mathcal{F}|$ the number of graphs in $\mathcal{F}$ and we write $V(\mathcal{F})=\bigcup_{G\in \mathcal{F}}V(G)$.

Throughout the paper, we will omit floor and ceiling signs when they are not essential. Also, we use standard hierarchy notation, that is, we write $a\ll b$ to denote that given $b$ one can choose $a_0$ such that the subsequent arguments hold for all $0<a\le a_0$. All logarithms except labeled one are natural.

\subsection{Sublinear expander}
Koml\'{o}s and Szemer\'{e}di \cite{KS1,KS} introduced a notion of sublinear expander that is a graph in which any subset of vertices of reasonable size expands by a sublinear factor. Significant progress in several seemingly unrelated long-standing open problems has been established by using techniques derived in the notion of sublinear expander. Remarkable applications include the resolution of Erd\H{o}s and Hajnal's odd cycle problem from 1966 \cite{odcy} and a series of breakthroughs concerning cycle decomposition (see \cite{Mon-cycle}), graph Ramsey goodness (see \cite{HHKL}), graph minor (see \cite{minorex}) and so on. For more recent applications of sublinear expanders to various problems in Extremal Combinatorics, we refer readers to a comprehensive survey of Letzter \cite{letz}.

Koml\'{o}s and Szemer\'{e}di \cite{KS} showed that every graph $G$ contains a sublinear expander almost as dense as $G$. Haslegrave, Kim and Liu \cite{Hasle2} slightly
generalised this to a robust version.
\begin{definition}[\cite{Hasle2,KS1,KS}]\label{def_sub}
Let $\varepsilon_1>0$ and $k\in \mathbb{N}$. A graph $G$ is an \emph{$(\varepsilon_1,k)$-robust-expander} if for any $X\subseteq V(G)$ of size $\tfrac{k}{2}\leq |X|\leq \tfrac{2|V(G)|}{3}$\ft{the coefficient $2/3$ makes no difference and can be easily verified to be valid by going through their proof in \cite{Hasle2}}, and any subgraph $F\subseteq G$ with $e(F)\leq d(G)\rho(|X|,\varepsilon_1,k)\cdot|X|$, we have
\begin{equation*}
|N_{G\setminus F}(X)|\geq \rho(|X|,\varepsilon_1,k)\cdot|X|
\end{equation*}
where
\[
\rho(x,\varepsilon_1,k):=\begin{cases}
		0 & \text{if }x<k/5,\\
		\eps_{1}/\log^2(15x/k) & \text{if }x\ge k/5.
	\end{cases}
\]
\end{definition}

\noindent
We often say such $G$ is an \emph{$(\eps_1,k)$-expander} if  $e(F)=0$. For simplicity, we write $\rho(x)$ for $\rho(x,\varepsilon_1,k)$. Note that $\rho(x)$ is a decreasing function when $x\geq \tfrac{k}{5}$.

%\begin{lemma}[\cite{KS}]\label{admin}
%There exists $\varepsilon_1>0$ such that the following holds for every $k>0$. Every graph $G$ has an $(\varepsilon_1,k)$-expander $H$ with $d(H)\geq \tfrac{d(G)}{2}$ and $\delta(H)\geq \tfrac{d(H)}{2}$.
%\end{lemma}
\begin{lemma}[\cite{Hasle2,KS}]\label{expander}
There exists $\eps>0$ such that the following holds for every $k>0$. Every graph $G$ has an $(\eps, k)$-robust-expander $H$ with $d(H)\ge d(G)/2$ and $\delta(H)\ge d(H)/2$.
\end{lemma}

%The `moreover' part can be easily obtained by going through their proof in \cite{KS}, though it is not explicitly stated in the original lemma.

Expanders are usually highly connected sparse graphs with certain expansion properties. A key property of the subliner expanders that we shall use is to connect vertex sets with a short path whilst avoiding a reasonable-sized set of vertices.
\begin{lemma}[Robust connectivity, \cite{KS}]\label{distance}
Let $\varepsilon, k>0$. If $G$ is an $n$-vertex $(\varepsilon,k)$-robust-expander, then for any two vertex sets $X_1, X_2$ each of size at least $x\geq k$, and a vertex set $W$ of size at most $\tfrac{\rho(x)x}{4}$, there exists a path in $G-W$ between $X_1$ and $X_2$ of length at most $\tfrac{2}{\varepsilon}\log^3\Big(\tfrac{15n}{k}\Big)$.
\end{lemma}
%This notion is essential in the recent resolutions of several long-standing conjectures.

%\subsection{Proof of Proposition \ref{}}\label{proproof}
\subsection{Units and  Webs}

%\subsection{Outlines of main steps}

%\subsection{Reduction to robustly dense graphs}\label{sec:local-dense}

\label{web-unit}
\begin{definition}[unit]\label{defn:unit}
For $h_1,h_2,h_3\in \mathbb{N}$, a graph $F$ is an \emph{$(h_1,h_2,h_3)$-unit} if it contains distinct vertices $u$ (the \emph{core} vertex of $F$) and $x_1,\ldots,x_{h_1}$, and $F=\bigcup_{i\in[h_1]}(P_i\cup S_{i})$, where
\begin{itemize}
\item $\mathcal{P}=\bigcup_{i\in[h_1]}P_i$ is a collection of pairwise internally vertex-disjoint paths, each of length at most $h_3$, such that $P_i$ is a $u,x_i$-path, and
\item $\mathcal{S}=\bigcup_{i\in[h_1]}S_{i}$ is a collection of vertex-disjoint $h_2$-stars such that $S_{i}$ has center $x_i$ and $\bigcup_{i\in[h_1]}(V(S_{i})\setminus \{x_i\})$ is disjoint from $V(\mathcal{P})$.
\end{itemize}
\end{definition}
We call $S_{i}$ a \emph{pendent} star in the unit $F$ and every such path $P_i$ is a \emph{branch} of $F$. Define the \emph{exterior} $\mathsf{Ext}(F):=\bigcup_{i\in[h_1]}(V(S_{i})\setminus \{x_i\})$ and \emph{interior} $\mathsf{Int}(F):=V(F)\setminus \mathsf{Ext}(F)$.
\begin{definition}[web]\label{defn:web}
For $h_1,h_2,h_3,h_4,h_5\in \mathbb{N}$, a graph $W$ is an \emph{$(h_4,h_5,h_1,h_2,h_3)$-web} if it contains distinct vertices $v$ (the \emph{core} vertex of $W$), $u_1,\ldots,u_{h_4}$, and $W=\bigcup_{i\in[h_4]}(Q_i\cup F_{i})$, where
\begin{itemize}
\item $\mathcal{Q}=\bigcup_{i\in[h_4]}Q_i$ is a collection of pairwise internally vertex-disjoint paths such that each $Q_i$ is a $v,u_i$-path of length at most $h_5$.
\item $\mathcal{F}=\bigcup_{i\in[h_4]}F_{i}$ is a collection of vertex-disjoint $(h_1,h_2,h_3)$-units such that $F_{i}$ has core vertex $u_i$ and $\bigcup_{i\in[h_4]}(V(F_{i})\setminus \{u_i\})$ is vertex-disjoint from $V(\mathcal{Q})$.
\end{itemize}
\end{definition}
%\begin{definition}[web]\label{defn:web}
%For $h_0,h_1,h_2,h_3\in \mathbb{N}$, a graph $W$ is an \emph{$(h_0,h_1,h_2,h_3)$-web} if it contains distinct vertices $v$ (the \emph{core} vertex of $W$), $u_1,\ldots,u_{h_0}$, and $W=\bigcup_{i\in[h_0]}(Q_i\cup F_{i})$, where
%\begin{itemize}
%\item $\mathcal{Q}=\bigcup_{i\in[h_0]}Q_i$ is a collection of pairwise internally vertex-disjoint paths such that each $Q_i$ is a $v,u_i$-path of length at most $h_3$.
%$(h_1,h_2,h_3)$-units such that $F_{i}$ has core vertex $u_i$ and $\bigcup_{i\in[h_0]}(V(F_{i})\setminus \{u_i\})$ is vertex-disjoint from $V(\mathcal{Q})$.
%\end{itemize}
%\end{definition}

We call each $Q_i$ a \emph{branch} and call the branches inside each unit $F_i$ the \emph{second-level} branches of $W$. Similarly define the \emph{exterior} $\mathsf{Ext}(W):=\bigcup_{i\in[h_0]}\mathsf{Ext}(F_{u_i})$, and the \emph{interior} $\mathsf{Int}(W):=V(W)\setminus \mathsf{Ext}(W)$ and additionally define \emph{center} $\mathsf{Ctr}(W):=V(\mathcal{Q})$.
\begin{figure}[H]
 \begin{center}
   \includegraphics[scale=0.6]{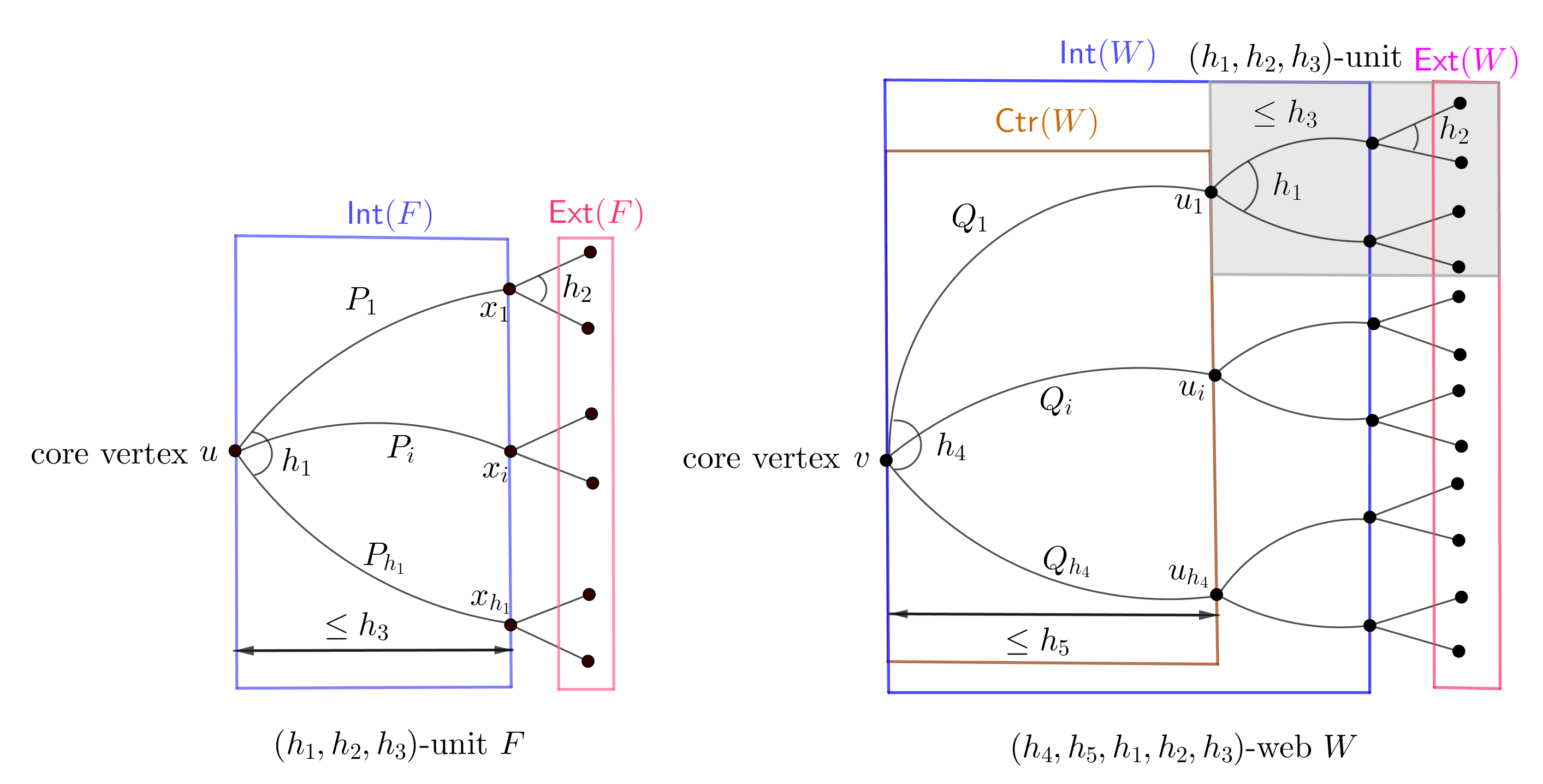}\\
   \caption{unit and web}
\label{webdraw}
 \end{center}
\end{figure}

\section{Main lemmas and overviews}\label{pf1}
In this section we give an outline of our proof, and compare with previous approaches.
%Note that by Lemma~\ref{expander}, $G$ contains a subgraph with expansion properties.
To construct a large clique subdivision as in Theorem \ref{main1}, we first apply Lemma~\ref{expander} to obtain in $G$ a subgraph with expansion properties. Unlike in \cite{crux, LiuC4}, the construction we use differs by comparing the value of $C_{\alpha}(G)$ with the density of the subgraph. We first consider the case when $C_{\beta}(G)$ is larger than $d^2\log^{O(1)} n$ (see Lemma~\ref{dense1}) for some $\beta$ slightly larger than $\alpha$ given as in Theorem \ref{main1}. Here we often call it a \textit{non-extremal} case as it is far apart from the space barrier. %\textit{extremal construction}: the $d/1000$-blow-up $G$ of a $d$-vertex $1000$-regular expander has $C_{\alpha}(G)=\Theta(d^2)$ and the largest clique subdivision has order $O(d/\sqrt{\log d})$.

\begin{lemma}[Non-extremal]\label{dense1}
Suppose $\tfrac{1}{n}, \tfrac{1}{d}\ll \tfrac{1}{T}, c\ll \beta,\eps \le\tfrac{1}{150}$ and $k\in \mathbb{N}$ has $k\le \min\{d^2,\tfrac{n}{T}\}$ and $d\ge \log^{1000}n$. Let $G$ be an $n$-vertex $(\eps, k)$-robust-expander with $\delta(G)\ge \tfrac{1}{2}d(G)\ge \tfrac{1}{16}\beta  d$ and $m=\log^4\tfrac{n}{k}$. If $\tfrac{C_{\beta}(G)}{\log C_{\beta}(G)}\ge d^2m^{100}$, then $G$ contains a $K_{cd}$-subdivision.
\end{lemma}
The case when $C_{\beta}(G)$ is close to $d^2$ (within a polylogarithmic factor) is most involved and the bulk of the work is to handle this, as the extremal construction (space barrier) may appear. %which essentially requires a clique subdivision of order $\Omega(\sqrt{\tfrac{C_{\beta}(G)}{\log C_{\beta}(G)}})$.
To achieve this, we choose a smallest $\beta$-crux $H\<G$ and attempt to find desired clique subdivisions within $H$ instead. We will again use Lemma~\ref{expander} to find an $(\eps, k)$-robust-expander $H^*\<H$ for a different value of $k$ and then utilize the following lemma. %reduce the problem to finding desired clique subdivisions inside a smallest $\beta$-crux, using the . with average degree at least $\sqrt{v(H)}/\log^{2500}v(H)$. To achieve this, we ...
%To overcome
%
\begin{lemma}[Extremal]\label{dense2}
Suppose $\tfrac{1}{n}, \tfrac{1}{d}\ll c\ll \beta,\eps\le\tfrac{1}{150}$ and $p,k\in \mathbb{N}$ satisfy $k=\min\{\tfrac{d^2}{\log^pd}, \tfrac{n}{\log ^pn}\}$ for $p\geq 1000$. Let $G$ be an $(\eps, k)$-robust-expander with $v(G)\le n$ and $\delta(G)\ge \tfrac{1}{2}d(G)\ge \tfrac{1}{16}\beta d$. If $d\ge \sqrt{\tfrac{n}{\log^p n}}$ and $C_{\beta}(G)\ge k\log k$, then $G$ contains a $K_{ct}$-subdivision, where $t=\min\{d,\sqrt{\tfrac{C_{\beta}(G)}{\log C_{\beta}(G)}}\}$.
\end{lemma}
We shall use Lemma~\ref{dense2} to cover the (extremal) case when $d\ge \sqrt{\tfrac{v(H)}{\log^p v(H)}}$ and $C_{\alpha}(G)\ge k\log k$. To tackle the space barrier, we need to bound the value of $k$ in Lemma~\ref{dense2} so as to robustly find short subdivided paths of nearly optimal length $O(\log n)$ (see Section~\ref{sec_3.1} for more details). %in repeated  applications of Lemma~\ref{distance}.

The remaining (non-extremal) case is now well understood, due to the work of Im, Kim, Kim and Liu \cite{crux}, and the following result suffices to give the required subdivision when $C_{\alpha}(G)<k\log k$, though the bound on $t$ is not optimal in general. %
\begin{lemma}[\cite{crux}]\label{lem_dense}
Suppose $\tfrac{1}{n} \ll \tfrac{1}{T} \ll \eps <\tfrac{1}{100}$ and $d \geq \exp(\log^{1/6} n)$.
If an $n$-vertex $(\eps, \eps d)$-robust-expander $G$ has average degree at least $\tfrac{d}{2}$ and minimum degree $\delta(G) \geq \tfrac{d}{4}$, then $G$ contains a $K_t$-subdivision where $t= \min\{\tfrac{d}{\log^{60}d}, \tfrac{\sqrt{n}}{(\log(Tn/d))^{11}}\}$.
\end{lemma}

It will be of independent interests to see whether one can find a clique subdivision of order linear in the average degree whenever the host expander is slightly sparser than the extremal construction (i.e.  $d \leq n^{1/2-o(1)}$). This is partially confirmed as follows.

\begin{lemma}\label{lem_sparse}
Suppose $\tfrac{1}{d} \ll c \ll \eps < \tfrac{1}{100}$ and $d \leq \sqrt[3]{\tfrac{n}{\log^{100}n}}$. If an $n$-vertex $(\eps, \eps d)$-expander $G$ has average degree at least $\tfrac{d}{2}$ and minimum degree $\delta(G) \geq \tfrac{d}{4}$, then $G$ contains a $K_{cd}$-subdivision.
\end{lemma}

Lemma~\ref{lem_sparse} actually strengthens a result in \cite{crux} where $d \leq \exp(\log^{1/6} n)$ (see Lemma \ref{ebound}) is proved, and its proof strategy essentially follows from that of Lemma~\ref{dense1}. We will add its proof in Appendix \ref{im3} for completeness.% As showed in Lemma~\ref{lem_dense}, this is asymptotically confirmed up to a polylogarithmic factor.

%The sparse case (see Lemma~\ref{lem_sparse}) is handled in a recent result of Im, Kim, Kim and Liu \cite{crux}.

%(see Lemmas \ref{lem_dense},\ref{dense1} and \ref{dense2}). Throughout the proof we always assume...

%\begin{lemma}[Extra]\label{lem_sparse1}\todo{missed case!}
%Suppose $\tfrac{1}{n}, \tfrac{1}{d}\ll c\ll \beta, \eps\ll\tfrac{1}{5}$ and $x\in \mathbb{N}$ has $x\geq 100$. Let $G$ be an $n$-vertex $(\eps,\eps d^2)$-expander with $\delta(G)\geq \tfrac{1}{16}\beta d$ and $m=\log^4\tfrac{n}{d^2}$.  If $d\ge \sqrt{\tfrac{n}{\log^p n}}$ and $C_{\beta}(G)\le\red{\tfrac{1}{5}d^2}$, then $G$ contains a $K_{ct}$-subdivision.
%\end{lemma}

With the four lemmas, we are now ready to prove Theorem \ref{main1}.
\begin{proof}[Proof of Theorem \ref{main1}]
We choose $\beta=\sqrt{4\alpha}\le \tfrac{1}{150}$ and $\tfrac{1}{d}\ll c\ll \tfrac{1}{T} \ll \beta,\eps$, and let $G$ be a graph with $d(G)=d$. Let
\[t=\min\Big\{d,\sqrt{\tfrac{C_{\alpha}(G)}{\log C_{\alpha}(G)}}
\Big\}.\] By Lemma~\ref{expander}, we have an $(\eps,\eps d)$-robust-expander $G_0\<G$ with $d(G_0)\ge \tfrac{d}{2},\delta(G_0)\geq \tfrac{d}{4}$ and let \[n=|V(G_0)|, ~m=\log^4\tfrac{n}{\eps d}.\]
Throughout the proof we assume that $\eps d<\tfrac{n}{T}$, otherwise by a result of Fox and Sudakov \cite{depran}, we can get a $K_{c\sqrt{n}}$-subdivision, as desired.
If $d \leq  \sqrt[3]{\tfrac{n}{\log^{100}n}}$ or $d \geq  \sqrt[3]{\tfrac{n}{\log^{100}n}}\geq \log^{1000}n$ and $\tfrac{C_{\beta}(G_0)}{\log C_{\beta}(G_0)}\ge d^2m^{100}$, then using Lemma~\ref{lem_sparse}, or Lemma~\ref{dense1} with $k=\eps d$, respectively, we obtain that $G_0$ (also $G$) contains a $K_{cd}$-subdivision as desired. It remains to consider the case $C_{\beta}(G_0)<d^2m^{100}\log C_{\beta}(G_0)<d^2\log^{500} n$ and $d> \sqrt[3]{\tfrac{n}{\log^{100}n}}$, which implies
\begin{align}\label{eq1}
d>\max\Big\{ \sqrt[3]{\tfrac{n}{\log^{100}n}}, \sqrt{\tfrac{C_{\beta}(G_0)}{\log^{500} n}}\Big\}.
\end{align}

Let $H\<G_0$ be a minimum $\beta$-crux of $G_0$. Then $v(H)=C_{\beta}(G_0)$ and $d(H)\ge \beta d(G_0)\ge \tfrac{1}{2}\beta d$. Taking $k=\min\{\tfrac{d^2}{\log^{2500}d}, \tfrac{v(H)}{\log^{2500}v(H)}\}$, we apply Lemma~\ref{expander} to obtain an $(\eps, k)$-robust-expander $H^*\<H$ with $d(H^*)\ge \tfrac{1}{4}\beta d$ and $\delta(H^*)\ge \tfrac{1}{2}d(H^*)\ge \tfrac{1}{8}\beta d$. By \eqref{eq1} it follows that
\[\log v(H)\ge \log d(H^*)\overset{\eqref{eq1}}{>} \tfrac{1}{4}\log n~~\text{and thus}~~d\overset{\eqref{eq1}}{>} \sqrt{\tfrac{v(H)}{\log^{500} n}}\ge\sqrt{\tfrac{v(H)}{\log^{2500} v(H)}}. \]
If $C_{\beta}(H^*)\ge k\log k$, then applying Lemma~\ref{dense2} on $H^*$ with $n=v(H)$ and $p=2500$, we obtain a $K_{ct_{\ref{dense2}}}$-subdivision in $H^*$ (also in $G$) with \[t_{\ref{dense2}}=\min\Big\{d,\sqrt{\tfrac{C_{\beta}(H^*)}{\log C_{\beta}(H^*)}}\Big\}\ge \min\Big\{d,\sqrt{\tfrac{C_{\alpha}(G)}{\log C_{\alpha}(G)}}\Big\}=t,\] where the inequality follows since $\beta=\sqrt{4\alpha}$ and thus
\[C_{\beta}(H^*)=\min\{|H'|: H'\<H^*, d(H')\ge \beta d(H^*)\ge \tfrac{1}{4}\beta^2 d= \alpha d\}\ge C_{\alpha}(G).
\]
Otherwise, we have $C_{\beta}(H^*)< k\log k\le \min\{\tfrac{d^2}{\log^{1500}d}, \tfrac{v(H)}{\log^{1500}v(H)}\}\le \min\{\tfrac{d^2}{\log^{1500}d}, \tfrac{n}{\log^{1500}n}\}$. This together with the fact $ C_{\beta}(H^*)\ge C_{\alpha}(G)$ implies that
\[t=\min\Big\{d,\sqrt{\tfrac{C_{\alpha}(G)}{\log C_{\alpha}(G)}}
\Big\}=\sqrt{\tfrac{C_{\alpha}(G)}{\log C_{\alpha}(G)}}\le \sqrt{k}\le \min\Big\{\tfrac{d}{\log^{60}d}, \tfrac{\sqrt{n}}{\log^{11}n}\Big\}\le t_{\ref{lem_dense}}.\]
Now we recap the assumption $d> \sqrt[3]{\tfrac{n}{\log^{100}n}}>\exp(\log^{1/6}n)$  and make use of Lemma~\ref{lem_dense} on $G_0$ to obtain a $K_{ct}$-subdivision as desired.
\end{proof}

\subsection{Proof Sketch of Lemmas}\label{sec_3.1}
We will give a unified approach for Lemma~\ref{dense1} and \ref{dense2} via Lemma~\ref{webs}, where the majority of the technical constructions are stated in a combined way.
\begin{lemma}\label{webs}
Suppose $\tfrac{1}{n}, \tfrac{1}{d}\ll \tfrac{1}{T}, c\ll\beta,\eps \le \tfrac{1}{150}$ and $x,k\in \mathbb{N}$ with $x>13, k\leq\tfrac{n}{T}, \ d\geq\log^{25x} n$. Let $G$ be an $n$-vertex $(\eps, k)$-robust-expander with $d(G)\ge d$, $\delta(G)\geq \tfrac{d}{2}$ and given $t$ satisfying \[ t=\min\left\{d, \sqrt{\tfrac{C_{\beta}(G)}{\log C_{\beta}(G)}}\right\}, \ k\le t^2m,\ t^2m^{4x} \leq \tfrac{1}{3}C_{\beta}(G),~~\text{where}~~ m= \log^4 \tfrac{n}{k}.\] Then $G$ contains a $K_{ct}$-subdivision.%$G$ contains $2t$ internally vertex-disjoint \red{$(ct, 3m+\ell, m^{x}, \tfrac{t}{40}, m+2)$-webs}, where  $\ell=\log \left(\tfrac{C_\beta(G)}{d}\right)$.
\end{lemma}

As mentioned in the introduction, the main tasks are to
\stepcounter{propcounter}
\begin{enumerate}[label = ({\bfseries \Alph{propcounter}\arabic{enumi}})]
   \item\label{build} build $\Omega(t)$ webs that have disjoint interiors, and then
   \item\label{connect} connect theirs exteriors (each of size essentially $t^2m^{O(1)}$ via pairwise disjoint short paths whilst avoiding the vicinity of core vertices.
\end{enumerate}
For \ref{build},
previous approaches for building units or webs usually proceed by linking many disjoint stars and then use averaging arguments, which can only produce web structures with sublinear in $d(G)$ many branches (e.g. in \cite{crux}, while in \cite{LiuC4}, they can obtain a linear bound under the $C_4$-free condition). This much weaker bound only leads them to build a clique subdivision of order sublinear in $d(G)$ (this is obviously not tight when $t=d(G)$). To overcome this, the first insight in this paper is a new perspective we bring to the $\alpha$-crux number. \medskip

\noindent\textbf{$(D,\mu)$-dense.} We use a notion of $(D,\mu)$-dense as follows.
%(use Lemma~\ref{sub}).
 %follow the strategy in \cite{H-sub}
%: first find a sufficient number of disjoint stars and then greedily connect them using Lemma~\ref{distance}. and

\begin{definition}\cite{GeneralH}
A graph $G$ is called \emph{$(D,\mu)$-dense} if
for every $W\subseteq V(G)$ with $|W|< D$, we have $d(G-W)\geq \mu d(G)$.
\end{definition}
\noindent Technically, we employ the following lemma, which allows us to iteratively build a web which avoids the vicinity of previous core vertices of webs built so far.

\begin{lemma}\label{prop}
Let $G$ be a graph with $\delta(G)\geq\tfrac{1}{2}d(G)$. Then
\stepcounter{propcounter}
\begin{enumerate}[label = ({\bfseries \Alph{propcounter}\arabic{enumi}})]
   \rm\item\label{robust}  \textit{for any $0<\alpha\le\tfrac{1}{5}$, $G$ is $(\tfrac{C_{\alpha}(G)}{2},\tfrac{1-3\alpha}{4})$-dense.}

    \rm\item\label{kexpan} \textit{for any $K,\alpha>0$ with $\alpha \le \tfrac{1}{2K+2}$, every $X\subseteq V(G)$ with $|X|\leq \tfrac{C_\alpha (G)}{K+1}$ satisfies $N_G(X)\geq K|X|$.}
\end{enumerate}
\end{lemma}
\noindent For now, an immediate consequence of \ref{robust} is that it
allows us to iteratively find large stars with up to $\tfrac{C_{\alpha}(G)}{2}$ vertices in total (even additionally avoiding the interiors of previous webs). But this is still far from clear whether we can build a web as promised in \ref{build}. Here we bring another new perspective to the $\alpha$-crux number. Instead of averaging arguments, a crucial ingredient in our construction is that we can robustly grow a ball around a fixed vertex (up to the size roughly $\tfrac{C_{\alpha}(G)}{3}$, see Lemma~\ref{ball}) so as to further connect up to many disjoint stars by disjoint paths. This indicates why our main proof is divided into three cases by distinguishing the value of the $\alpha$-crux number with the density $d(G)$. Also, to make this happen, we need the following notion formulated in \cite{LiuC4} to iteratively connect from a fixed vertex via disjoint paths.%We will use the (strong) expansion property from the $\alpha$-crux (as in \ref{kexpan}) to expand around the fixed vertex. (more details?)
%$i)$ first robustly grow a ball around a fixed vertex of large degree (core vertex for the desired web), and $ii)$ then connect the core vertex with many disjoint units, using the (strong) expansion property from the $\alpha$-crux (as in \ref{kexpan}). % to construct the larger webs.

\begin{definition}[\cite{LiuC4}]
Given a graph $G$ and $W\subseteq V(G)$, we say that paths $P_1,\ldots,P_t$, each starting with a vertex $v$ and contained in the vertex set $W$, are \emph{consecutive shortest paths} from $v$ in $W$ if for each $i$ $(1\leq i\leq t)$, the path $P_i$ is a shortest path between its endpoints in the set $W\setminus(\bigcup_{j<i}V(P_j))\cup\{v\}$.
\end{definition}

\noindent The following lemma allows us to robustly grow a ball.
\begin{lemma}\label{ball}
Suppose $\tfrac{1}{d}\ll c,\beta\le\tfrac{1}{5}$ and let $G$ be a graph with $\delta(G)\geq \tfrac{d(G)}{2}\geq 5cd$. Then for any integer $D\le \tfrac{1}{3}C_{\beta}(G)$, there exists a positive integer $\ell\le\log_2 (\tfrac{2D}{\delta(G)})$ such that given any vertex $v$ and consecutive shortest paths $P_1,\dots, P_t$, $t\leq cd$, from $v$ in $B_{G}^{\ell}(v)$, writing $U=\bigcup_{i\in t}V(P_i)$, we have $|B_{G-(U\backslash\{v\})}^{\ell}(v)|\geq D$.
\end{lemma}

However, building up to now, a subtle fact is that once we carry out building a web by using Lemma \ref{ball} to expand from a fixed vertex, the resulting web would witness $\Omega(t)$ consecutive shortest paths whose union has size $\Omega(t\ell)$. Thus to make \ref{build} successful, in each iteration (at least for the last few iterations), we need to avoid $\Omega(t^2\ell)$ vertices from previous webs. Unfortunately, the space barrier tells that $C_{\beta}(G)=\Omega(n)$ and $\ell=\Omega(\log n)$, and thus $\Omega(t^2\ell)=\Omega(n)$. There are too many (up to $\Omega(n)$) internal vertices to avoid just by applying the current connection tools such as Lemma~\ref{distance}. It seems that we need a stronger connectivity (than Lemma~\ref{distance}) for building short paths, and this appears surprisingly difficult in sublinear expanders.\medskip

\noindent\textbf{Connecting through a random vertex set.} To overcome this, we modify our strategy for \ref{build} as follows: we first partition
$V(G)=V_1\cup V_2$ by including each vertex independently in $V_i$ uniformly at random. Then in each iteration of construction, we
\stepcounter{propcounter}
\begin{enumerate}[label = ({\bfseries \Alph{propcounter}\arabic{enumi}})]
   \item\label{build1} first apply Lemma~\ref{ball} to robustly grow a ball $B$ (of size roughly $\tfrac{C_{\alpha}(G)}{3}$) within $V_1$ which is disjoint from the interiors of previous webs; and meanwhile build many disjoint units in $V_2$;
   \item\label{connect1} then connect the ball up to many distinct units via pairwise disjoint short paths \textit{through} $V_2$ whilst avoiding the interiors of previous webs; Meanwhile we find consecutive shortest paths so as to extend every connection up to the core vertex of the ball.
\end{enumerate}
Here a path \textit{through} $V_2$ is one whose internal vertices all lie in $V_2$. To carry out \ref{build1}, our new perspective on the crux number is critical. Note that the union of all consecutive shortest paths within previous webs built as in \ref{connect1} has size at most $t^2\ell<C_{\beta}(G)/2$. By iteratively applying Lemma~\ref{prop} \ref{robust} on $G[V_1]$ (minus the interiors of previous webs), we will find a subgraph with average degree linear in $d(G)$. To ensure a ball of the desired size in the subgraph, we recap the definition of the $\alpha$-crux number and use the following easy fact:
\begin{quote}
\textit{Given a graph $G$ and a subgraph $H\<G$ with $d(H)\ge \tfrac{1}{k} d(G)$ for some $k\in \mathbb{N}$, we obtain that $C_{k\alpha}(H)\ge C_{\alpha}(G)$ for any $0<\alpha<\tfrac{1}{k}$.}
\end{quote}
Thus such a ball always exists from Lemma~\ref{ball}. To build disjoint units in $V_2$, again by using Lemma~\ref{prop} \ref{robust} on $G[V_2]$, our approach follows the usual averaging arguments as in \cite{LiuC4}. \medskip

Note that each connection in \ref{connect1} automatically avoids $O(t^2\ell)=O(C_{\beta}(G))$ internal vertices from consecutive shortest paths of previous webs. This turns out to be very efficient as all other internal vertices sum up to a size $t^2m^{O(1)}=o(C_{\beta}(G))$ (valid in the proof of Lemma~\ref{webs}) that is easy to avoid.  To make \ref{connect1} successful, it is desired that $G[V_2]$ is also a sublinear expander and thus we can apply Lemma~\ref{distance}.  However it is surprisingly difficult to verify whether with positive probability, $G[V_2]$ is a sublinear expander as in Definition \ref{def_sub} (see \cite{Mon-cycle} for a detailed explanation why the union bound is inefficient).
To get around this, we adopt an idea of Buci\'{c} and Montgomery \cite{Mon-cycle} where they bring many new perspectives on robust sublinear expansion in Definition~\ref{def_sub} with useful applications in decompositions of expanders into small ones. To be more precise, they show that very likely a random set $V\< V(G)$ inherits a slightly weak expansion property saying that one can expand from any reasonable-sized vertex set (not necessarily lie inside $V$) in a few more steps to reach up to $\tfrac{|V|}{2}$ vertices within $V$, even after additionally deleting superlinearly many edges.

To aid our proof, we develop a random variant of Lemma~\ref{distance}, where we also allow the deletion of sublinearly many vertices.

%\red{Our approach differs from previous works on building subdivisions. By previous idea, we could not avoid centers of all webs to form the desired subdivision. To aid our proof, we adopt a more elusive approach to construct webs. First, we divide the graph into two parts randomly and grow up a ball from one part $V_1$ via consecutive shortest paths and then build many units in the other part $V_2$. Then applying Lemma \ref{random_distance} to greedily get webs $\mathcal{W}$ such that most of the vertices of $\mathsf{Ctr}(\mathcal{W})$ lie in $V_1$ and the remaining vertices of $\mathcal{W}$ lie in $V_2$. Last, we just link webs pairwise inside $V_2$. So it is possible to finish the connections while avoiding all vertices in $\mathsf{Ctr}(\mathcal{W})$ (in $V_1$) by Lemma \ref{random_distance}. }

\begin{lemma}[Robust Connecting Lemma]\label{random_distance}
Suppose $\tfrac{1}{n},\tfrac{1}{d}\ll \tfrac{1}{k}\ll \varepsilon<1$ and $G$ is an $n$-vertex $(\varepsilon,k)$-robust-expander with average degree $d\ge \log^{100} n$, and $V\subseteq V(G)$ is a random subset chosen by including each vertex independently at random with probability $\tfrac{1}{2}$. Let $m=\log^4\tfrac{n}{k}$. Then, with probability $1-o(\tfrac{1}{n})$, for any two vertex sets $X_1, X_2$ each of size at least $x\geq km^9$, and a vertex set $W$ of size at most $\tfrac{x}{m^{11}}$, there exists a path through $V$ of length at most $2m^2$ between $X_1$ and $X_2$, which is vertex disjoint from $W$.
\end{lemma}

To sum up, this will allow us to build webs as in step \ref{build} by recursively carrying out \ref{build1}-\ref{connect1}. Now with the aid of Lemma~\ref{random_distance}, the final connections in \ref{connect} will be easy by linking their exteriors via disjoint paths through $V_2$. \medskip
%It worth a remark that the combination of our new perspectives on $\alpha$-crux number and the modified strategy of connecting through the random vertex set $V_2$, plays an important role in \ref{build} and \ref{connect}, and we hope to find more applications.\medskip

To end this subsection, we give short proofs of Lemma~\ref{prop} and Lemma~\ref{ball}, while the proof of Lemma~\ref{random_distance} closely follows the strategy in \cite{Mon-cycle} and thus deferred to Appendix~\ref{sec4.1}.
\begin{proof}[Proof of Lemma~\ref{prop}]
For part \ref{robust}, let $G$ be a graph with $d(G):=d$ and write $k=\lceil\tfrac{C_{\alpha}(G)}{2}\rceil$. Fix any subset $W\subseteq V(G)$ with $|W|=k-1$. Then $|V(G)\setminus W|>k-1$ and let $n=|V(G)\setminus W|$. We arbitrarily take a subset $Y\<V(G)\setminus W$ of size $k-1$ and let $Z=V(G)\setminus(W\cup Y)$. Note that $|X\cup Y|<C_{\alpha}(G)$, and thus by the definition of $C_\alpha(G)$ we have $e(X\cup Y)<(k-1)\alpha d$ and $e(Y)<\tfrac{1}{2}(k-1)\alpha d$. Then \[e(Y,Z)= \sum_{v\in Y} d(v)-2e(Y)-e(X,Y)\ge \tfrac{1}{2}(k-1)d-\tfrac{3}{2}(k-1)\alpha d=(k-1)\tfrac{1-3\alpha}{2}d\ge (k-1)\alpha d.\] This gives $|Z|\ge k$ as otherwise we have $|Y\cup Z|\le 2(k-1)<C_{\alpha}(G)$ whilst $d(G[Y\cup Z])\ge \alpha d$, a contradiction.
Thus $n\ge 2k-1$ and by double counting the sums of $e(Y,Z)$ over all choices of $Y\<V(G)-W$, we have
\[
\tbinom{n}{k-1}(k-1)\tfrac{1-3\alpha}{2}d\le \sum_{|Y|=k-1} e(Y,Z)\le 2e(G-W)\tbinom{n}{k-2},
\]
which implies $d(G-W)=\tfrac{2e(G-W)}{n}\ge \tfrac{n-k+2}{n}\cdot\tfrac{1-3\alpha}{2}\cdot d\ge \tfrac{1-3\alpha}{4}d$ as $n\ge 2k-1$.

For part \ref{kexpan}, take a vertex set $X\subseteq V(G)$ with $|X|\leq \tfrac{C_\alpha (G)}{K+1}$. Suppose to the contrary that $N_G(X)< K|X|$, and we consider the subgraph $H:=G[X\cup N_G(X)]$. Thus
    \begin{equation*}
       d(H)\geq \tfrac{\delta(G)|X|}{|X\cup N_G(X)|}>\tfrac{1}{2K+2}d(G)\ge \alpha d(G).
    \end{equation*}
    However,
    \begin{equation*}
        v(H)< |X|+K|X|\leq C_\alpha(G),
    \end{equation*}
which contradicts to the definition of $C_\alpha(G)$.
\end{proof}

\begin{proof}[Proof of Lemma~\ref{ball}]
Given $\beta\le\tfrac{1}{5}$, we choose $\tfrac{1}{d}\ll c, \beta$  and let $G$ be with $\delta(G)\geq \tfrac{d(G)}{2}\geq 5cd$, and $P_1,\ldots,P_t$ be consecutive shortest paths from $v$ in $V(G)$. Let $U=\bigcup_{i\in [t]}V(P_i)$ and $F=G-(U \backslash\{v\})$. Recall that $D\le \tfrac{C_{\beta}(G)}{3}$. We shall show by induction on $p\geq1$ that, if $|B_{F}^p(v)|\leq \tfrac{C_{\beta}(G)}{3}$, then
\begin{equation}\label{inequality}
|N_F(B_{F}^p(v))|\geq |B_{F}^p(v)|.
\end{equation}
Also we will show that $|B_{F}^{1}(v)|\geq \tfrac{4}{5}\delta(G)$, which together with this inductive statement will prove the lemma. Actually,
we may take the conclusion for granted, and thus if $|B_{F}^p(v)|\leq \tfrac{C_{\beta}(G)}{3}$, then we have
\[|B_{F}^{p+1}(v)|=|B_{F}^{p}(v)|+|N_F(B_{F}^p(v))|\ge 2|B_{F}^{p}(v)|.
\]
Then there exists $\ell\le \log_2 (\tfrac{2D}{\delta(G)})$ for which $|B_{F}^{\ell}(v)|\ge 2^\ell |B_{F}^{1}(v)|\geq 2^{\ell} \cdot \tfrac{4}{5}\delta(G)\geq D$ as desired.

Thus, we only need to prove the inductive statement and $|B_{F}^{1}(v)|\geq \tfrac{4}{5}\delta(G)$. As the paths $P_i$ are also consecutive shortest paths from $v$ in $B^{\ell}_{G}(v)$, only the first $p+2$ vertices of each path $P_i$, including $v$, can belong to $N_G(B_{G-\cup_{j<i}(V(P_j)\backslash\{v\})}^p(v))$. Hence, only the first $p+2$ vertices of each of the path $P_i$, including the vertex $v$ can belong to $N_G(B_{F}^p(v))$. On the other hand, as we have at most $cd$ paths $P_i$, $|N_G(B_{F}^p(v))\cap U|\leq(p+1)cd$, so that
\begin{equation}\label{ballineq}
|N_{G-F}(B_{F}^p(v))|\leq(p+1)cd.
\end{equation}
In particular, when $p=0$, we have
\begin{equation}\label{less1}
|N_F(B_{F}^0(v))|=|N_F(v)|\geq d(v)-cd\geq \tfrac{4}{5}\delta(G).
\end{equation}
Hence, $|B_{F}^1(v)|\geq \tfrac{4}{5}\delta(G)\geq 4cd$.
Next we aim to prove (\ref{inequality}). When $p\geq 1$. Suppose that (\ref{inequality}) holds for all $p'$ with $0\leq p'<p$. By the assumption $|B_{F}^p(v)|\leq \tfrac{C_{\beta}(G)}{3}$, Proposition \ref{prop} \ref{kexpan} applied with $K=3/2$ and $\alpha=\beta$, implies that $|N_G(B_{F}^p(v))|\ge \tfrac{3}{2}|B_{F}^p(v)|$. Now by (\ref{ballineq}) and the fact that
\[|N_G(B_{F}^p(v))|=|N_{G-F}(B_{F}^p(v))|+|N_F(B_{F}^p(v))|,\]
it suffices to prove that
\begin{equation}\label{less2}
|B_{F}^p(v)|\ge 2(p+1)cd.
\end{equation}
By the induction hypothesis, we have $|B_{F}^{p'}(v)|=|B_{F}^{p'-1}(v)|+|N_F(B_{F}^{p'-1}(v))|\ge 2|B_{F}^{p'-1}(v)|$ for every $1\leq p'\le p$ and thus $|B_{F}^{p}(v)|\ge 2^{p-1}|B_{F}^{1}(v)|\ge2^{p-1}4cd\ge 2(p+1)cd$, finishing the proof.
\end{proof}

\subsection{Finish up: proofs of Lemmas \ref{dense1} and \ref{dense2}}\label{Thm3simi}

\begin{proof}[Proof of Lemma~\ref{dense1}]
We choose $\tfrac{1}{n}, \tfrac{1}{d}\ll \tfrac{1}{T}, c\ll\beta,\eps \le \tfrac{1}{150}$ and
let $G$ be an $n$-vertex $(\varepsilon, k)$-robust-expander with $\delta(G)\geq \tfrac{d(G)}{2}\geq \tfrac{\beta d}{16}$ and $k\le \min\{d^2,\tfrac{n}{T}\}$. Since $d^2m^{100}\leq \tfrac{C_{\beta}(G)}{\log C_{\beta}(G)}$ and $\beta\le \tfrac{1}{150}$, taking $t=\min\{\tfrac{\beta d}{8}, \sqrt{\tfrac{C_{\beta}(G)}{\log C_{\beta}(G)}}\}=\tfrac{\beta d}{8}$, we have
\begin{equation*}
   k\le d^2< t^2m, \ \  t^2m^{100}\leq \tfrac{\beta^2}{64}\cdot \tfrac{C_\beta(G)}{\log C_\beta(G)}\leq \tfrac{1}{3}C_\beta(G),
\end{equation*}
as $\tfrac{n}{k}$ and also $m$ are sufficiently large. Applying Lemma~\ref{webs} with $(k,t,d,x)=(k,\tfrac{\beta d}{8},\tfrac{\beta d}{8},25)$, we can obtain that $G$ contains a $K_{cd}$\footnote{Actually $c$ can arrive $\tfrac{1}{1200}$ when we take $\beta=\tfrac{1}{150}$.}-subdivision.
\end{proof}

\begin{proof}[Proof of Lemma~\ref{dense2}]
We choose $\tfrac{1}{n}, \tfrac{1}{d}\ll c\ll\beta,\eps \le \tfrac{1}{150}$ and
let $G$ be an $(\varepsilon, k)$-robust-expander with $\delta(G)\geq \tfrac{1}{2}d(G)\geq\tfrac{1}{16}\beta d $ and
\[d\geq \sqrt{\tfrac{n}{\log^p n}}, ~C_\beta(G)\geq k\log k,~k=\min\{\tfrac{d^2}{\log^p d}, \tfrac{n}{\log^p n}\},~\text{and let}~~m=\log^4\tfrac{n}{k}.\]
Note that by the assumption $d\geq \sqrt{\tfrac{n}{\log^p n}}$, we obtain that \[k=\min\{\tfrac{d^2}{\log^p d}, \tfrac{n}{\log^p n}\}> \tfrac{n}{\log^{2p} n}~~\text{and thus}~~m\le (\log\log n)^5.\]
Let $t=\min\{\tfrac{\beta d}{8},\sqrt{\tfrac{C_{\beta}(G)}{\log C_{\beta}(G)}}\}$. We shall find a $K_{ct}$-subdivision by Lemma~\ref{webs}.
From above, we want to claim that
\begin{equation}\label{ine1}
 k\leq t^2m, \ \ t^2m^{100}\le \tfrac{C_\beta(G)}{3}.
\end{equation}
Indeed, as $C_\beta(G)\geq k\log k$, we have
\begin{equation*}
    t^2m=\min\{\tfrac{\beta^2 d^2}{64},\tfrac{C_\beta(G)}{\log C_\beta(G)}\}m\geq \min\{\tfrac{\beta^2d^2}{64}, \tfrac{k\log k}{\log(k\log k)}\}m\geq k.
\end{equation*}
On the other hand, as $k>\tfrac{n}{\log^{2p}n}$ and $m<(\log \log n)^5$, we have
\begin{equation*}
    \log C_{\beta}(G)\geq \log(k\log k)>\log (\tfrac{n}{\log^{2p}n})>3m^{100}.
\end{equation*}
Hence,
\begin{equation*}
t^2m^{100}\le\tfrac{C_{\beta}(G)}{\log C_{\beta}(G)}m^{100}\le\tfrac{C_{\beta}(G)}{3}.
\end{equation*}
Consequently, (\ref{ine1}) holds. Finally, applying Lemma~\ref{webs} with $(k,t,d,x)=(k,t,\tfrac{\beta d}{8},25)$, we obtain that $G$ contains a $K_{ct}$-subdivision.
\end{proof}

%If $k=\tfrac{d^2}{\log^p d}$, which means that $\tfrac{d^2}{\log^p d}<\tfrac{n}{\log^p n}$, then $\tfrac{n}{k}=\tfrac{n\log^pd}{d^2}\leq \log^pd\cdot\log^pn$. In this case, we shall prove that
%\begin{equation}\label{ine3}
  %  \eps\tfrac{d^2}{\log^p d}\leq \tfrac{\beta^2t_1^2}{64}\log^{4x}(\tfrac{n\log^pd}{d^2})\leq \tfrac{C_\beta(G)}{3}.
%\end{equation}
%For (\ref{ine3}), the first inequality holds as $\tfrac{x}{\log x}$ is an increasing function, $C_\beta(G)\geq k$ and
%\begin{equation*}
 %   \tfrac{\beta^2}{64}(\log n+p\log\log d-2\log d)^{4x}\geq \log n-p\log\log n.
%\end{equation*}
%And the last inequality holds as
%\begin{equation*}
%\beta^2(7p\log\log d)^{4x}<\tfrac{1}{3}\log d\leq \tfrac{1}{3}\log C_{\beta}(G),
%\end{equation*}
%where $\log d>\log^{\tfrac{1}{6}}n$ implies that $\log\log n<6\log\log d$.

%If $k=\tfrac{n}{\log^p n}$, which means that $\tfrac{n}{\log^p n}<\tfrac{d^2}{\log^p d}$, then $\tfrac{n}{k}=\log^pn$. In this case, we shall prove that
%\begin{equation}\label{ine4}
 % \eps \tfrac{n}{\log^p n}\leq \tfrac{\beta^2t_1^2}{64}\log^{4x}(\log^pn)\leq \tfrac{C_\beta(G)}{3}.
%\end{equation}
%For (\ref{ine4}), the first inequality holds as
%\begin{equation*}
 % \tfrac{\beta^2}{64}(p\log\log n)^{4x}\geq \log\log\tfrac{n}{\log^pn},
%\end{equation*}
%and the last inequality holds as
%\begin{equation*}
%\beta^2(p\log\log n)^{4x}<\tfrac{1}{3}\log (\tfrac{n}{\log^p n})\leq \tfrac{1}{3}\log C_{\beta}(G).
%\end{equation*}

\section{Proof of Lemma~\ref{webs}}

\begin{proof}[Proof of Lemma~\ref{webs}]
Choose $\tfrac{1}{n}, \tfrac{1}{d}\ll \tfrac{1}{T}, c\ll\beta,\eps \le \tfrac{1}{150}$ and $x,k\in \mathbb{N}$ with $x>13, k\leq\tfrac{n}{T}, d\geq\log^{25x} n$. Let $G$ be an $n$-vertex $(\eps, k)$-robust-expander with $d(G)\ge d$, $\delta(G)\geq \tfrac{d}{2}$ and given $t\in \mathbb{N}$ satisfying
\[ t=\min\left\{d, \sqrt{\tfrac{C_{\beta}(G)}{\log C_{\beta}(G)}}\right\}, \ k\le t^2m, \ t^2m^{4x} \leq \tfrac{1}{3}C_{\beta}(G),~~\text{where}~~ m= \log^4 \tfrac{n}{k}.\]
Since $C_{\beta}(G)\ge \beta d$, we get that $t\ge m^{3x}$.
To build a $K_{ct}$-subdivision, we first partition $V(G)=V_1\cup V_2$, where each vertex is independently included in $V_i$, $i\in [2]$, uniformly at random. It follows from standard concentration methods that for each $i\in [2]$, $G_i:=G[V_i]$ has $d(G_i)\ge \tfrac{d}{3}$.

\medskip
\noindent \textbf{Phase $1$. Building $t$ webs that have disjoint interiors}.
\medskip

Let $\ell=\log_2(\tfrac{C_{\beta}(G)}{d})$. We need the following main result.
\begin{claim}\label{many_web}
The graph $G$ contains $t$ distinct $(4ct, \ell+m^{3}, 2m^x,\tfrac{t}{80},m^3)$-webs $\mathcal{W}_1,\ldots, \mathcal{W}_{t}$ such that they have disjoint interiors and each $\mathcal{W}_i$ satisfies $|\mathsf{Int}(\mathcal{W}_i)\cap V_1|\le 4ct\ell$ and $|\mathsf{Int}(\mathcal{W}_i)\cap V_2|\le 10ctm^{x+3}$.
\end{claim}
\begin{proof}[Proof of Claim~\ref{many_web}]
As sketched in Section~\ref{sec_3.1}, we shall proceed by first growing a ball inside $V_1$ and meanwhile building many disjoint units in $V_2$, whilst avoiding the interiors of previous webs. Suppose we have built $\mathcal{W}_1,\ldots, \mathcal{W}_{t'}$ as desired for some $0\le t'< t$, and let
\[W_i=\cup_{j\in[t']}\mathsf{Int}(\mathcal{W}_j)\cap V_i~\text{for}~i\in[2].\]
Then $|W_1|\le 4ct^2\ell$ and $|W_2|\le t^2m^{2x}$.
We first build many disjoint units in $V_2$ in the following, whose proof is postponed later.
\begin{subclaim}[build units]\label{web-unit1}
%Given any $W\subseteq V(G)$ with $|W|\leq t^2m^{2x}$,
The graph $G_2-W_2$ contains vertex-disjoint $(2m^x,\tfrac{t}{40},m^3)$-units $F_1,\ldots, F_{2tm^{x+12}}$.
\end{subclaim}
%Applying Subclaim \ref{web-unit1}, we get pairwise vertex-disjoint $(2m^{x},\tfrac{t}{20},m+2)$-units $F_{1},\ldots,F_{2tm^{x+1}}$ in $G_2-W_2$, and
Denote by $f_i$ the core vertex of each $F_i$ as above. % and $F=\{f_{1},\ldots,f_{2tm^{2x+1}}\}$.
Let \[W_3=\bigcup_{i\in [2tm^{x+12}]}V(F_i)~\text{ and}~ W_4=\bigcup_{i\in [2tm^{x+12}]}\mathsf{Int}(F_i).\] Then we have $|W_4|\leq 5tm^{2x+15}$, $|W_3|\leq \tfrac{t^2m^{2x+15}}{4}$.

Next, we shall grow a ball in $V_1-W_1$. Observe that $d(G_1)\ge \tfrac{1}{3}d$ and thus one can easily choose a subgraph $H\<G_1$ such that $\delta(H)\ge \tfrac{1}{2}d(H)\ge \tfrac{1}{2}d(G_1)$. Then by definition $C_{3\beta}(H)\ge C_{\beta}(G)$. Note that \[|W_1|\leq 4ct^2\ell<\tfrac{C_{\beta}(G)}{2}~\text{ as}~\ell<\log C_{\beta}(G)~\text{and}~ t\le \sqrt{\tfrac{C_{\beta}(G)}{\log C_{\beta}(G)}}.\] Applying Lemma~\ref{prop}~\ref{robust} to $H$, we get that $d(H-W_1)\geq \tfrac{d(H)}{10}\ge \tfrac{d}{30}$. Further choose a subgraph $H'\<H-W_1$ such that $\delta(H')\ge \tfrac{1}{2}d(H')\ge \tfrac{1}{2}d(H-W_1)\ge \tfrac{d}{60}$ and pick a vertex $v\in V(H')$ of degree at least $\tfrac{d}{30}$.
\begin{subclaim}
There is a $(4ct,\ell+m^3,m^x,\tfrac{t}{80},m^3)$-web centered at $v$, denoted $\mathcal{W}_{t'+1}$, which is internally disjoint from $W_1\cup W_2$, and $|\mathsf{Int}(\mathcal{W}_{t'+1})\cap V_1|\le 4ct \ell$, $|\mathsf{Int}(\mathcal{W}_{t'+1})\cap V_2|\le 10ctm^{x+3}$.
\end{subclaim}
 \begin{proof}
% Suppose to the contrary that $|B^{\ell}_{G_1}(v)|< \tfrac{C_{10\beta}(G_1)}{3}$. Observe that $|B^{1}_{G_1}(v)|\geq d(v)=\tfrac{d}{10}$. Applying Proposition \ref{Kexpan} with $K=3$\todo{means $\beta \ll 1/30$}, $|B^{i}_{G_1}(v)|\geq3|B^{i-1}_{G_1}(v)|$, whence
% \begin{equation*}
%     \tfrac{C_{10\beta}(G_1)}{3}>|B^{\ell}_{G_1}(v)|\geq 3^{\ell-1}|B^{1}_{G_1}(v)|\geq \tfrac{d\cdot 3^{\ell-1}}{10}.
% \end{equation*}
% Then $\ell \leq \log(\tfrac{C_{10\beta}(G_1)}{30d})+1<\log^3 (\tfrac{C_{10\beta}(G_1)}{300d})$, which contradicts to the choice of $\ell$.
%
To build the desired $(4ct, \ell+m^3, m^x, \tfrac{t}{80}, m^3)$-web, we shall proceed by finding $4ct$ internally vertex-disjoint paths $Q_1,\ldots,Q_{4ct}$ in $G$ from $v$ satisfying the following rules.%, where $\ell$ is a constant obtained from Lemma~\ref{ball} applied on $G_1$.
\stepcounter{propcounter}
\begin{enumerate}[label = ({\bfseries \Alph{propcounter}\arabic{enumi}})]
\rm\item\label{youtiao1} Each path $Q_i$ is a unique $v,f_{x_i}$-path, where $x_i\in[2tm^{x+12}]$, and furthermore $|V(Q_i)\cap V_1|\le \ell, |V(Q_i)\cap V_2|\le m^3$.
\rm\item\label{youtiao2} Each path does not contain any vertex in $W_2\cup W_4$ as an internal vertex.
\rm\item\label{youtiao3} The sequence of subpaths $Q_i[B^{\ell}_{H'}(v)]$ $(i\in[4ct])$ form consecutive shortest paths from $v$ in $B^{\ell}_{H'}(v)$.
\end{enumerate}

Assume that we have iteratively obtained a collection of shortest paths $\mathcal{Q}=\{Q_1,\ldots,Q_s\}$ $(0\leq s<4ct)$ as in \ref{youtiao1}-\ref{youtiao3}. Note that \ref{youtiao3} gives $s$ consecutive shortest paths $P_1,\ldots,P_s$ from $v$ in $B^{\ell}_{H'}(v)$, where $P_i=Q_i[B^{\ell}_{H'}(v)]$ and write $\mathcal{P}=\{P_1,\ldots,P_s\}$. Note that $C_{30\beta}(H')\geq C_{\beta}(G)$. Applying Lemma~\ref{ball} to $H'$ with $D=\tfrac{C_\beta(G)}{3}$ , we get
$$|B^{\ell}_{H'-\mathsf{Int}(\mathcal{Q})}(v)|=|B_{H'-\mathsf{Int}(\mathcal{P})}^{\ell}(v)|\geq  \tfrac{C_{\beta}(G)}{3}>t^2m^{4x-1}.$$
 We call a unit $F_i$ \textit{available} if its core vertex is not used as an endpoint of a path in $\mathcal{Q}$.
Let $U'$ be the leaves of all pendent stars in available units. Then we have
$$|U'|\geq 2m^{x}\cdot\tfrac{t}{40}(2tm^{x+12}-4ct)>\tfrac{t^2m^{2x+12}}{20}.$$
Note that
$$|W_2|+|W_4|+|V(\mathcal{Q})\cap V_2|\leq t^2m^{2x}+5tm^{2x+15}+4ctm^3<2t^2m^{2x}. $$
Applying Lemma~\ref{random_distance} on $G$ with $V_2,B_{H'-\mathsf{Int}(\mathcal{Q})}^{\ell}(v)$, $U'$, $W_4\cup W\cup (V(\mathcal{Q})\cap V_2)$ playing the roles of $V,X_1,X_2,W$, respectively, we can find a shortest path through $V_2$ of length at most $m^3$, say $Q$, joining a vertex $w$ from $B_{H'-\mathsf{Int}(\mathcal{Q})}^{\ell}(v)$ to a leaf $u_j\in \mathsf{Ext}(F_j)$ for some available unit $F_j$. One can easily find a $v,w$-path within $B^{\ell}_{H'-\mathsf{Int}(\mathcal{P})}(v)$, denoted as $P_{s+1}$,  and a $u_j,f_j$-path inside the unit $F_j$, denoted as $R_{s+1}$. Let $Q_{s+1}=P_{s+1}QR_{s+1}$. Then the paths $Q_1,\ldots,Q_{s+1}$ satisfy \ref{youtiao1}-\ref{youtiao3}. Repeating this for $s=0,1,\ldots,4ct$, yields $4ct$ paths $Q_1,\ldots, Q_{4ct}$ as desired.

For every $i\in [2tm^{x+12}]$, we say a pendent star in $F_i$ is \emph{overused} if at least $\tfrac{t}{80}$ (half of) leaves are used in $V(\mathcal{Q})$ and a unit is  \emph{bad} if at least $m^x$ pendent stars are overused. Then there are at most $\tfrac{4ctm^3}{t/80}=320cm^3<m^x$ overused pendent stars, and therefore there is no bad unit. Then removing from every unit the second-level branches each attached with overused pendant stars and combining corresponding paths $Q_i$ yield a $(4ct, \ell+m^3, m^x, \tfrac{t}{80}, m^3)$-web as desired.
\end{proof}
Up to now, the proof of Claim~\ref{many_web} is completed.
\end{proof}
\medskip
\noindent \textbf{Phase $2$. Building a $K_{ct}$-subdivision: linking webs via paths through $V_2$}.
\medskip

Let $\mathcal{W}_1, \ldots, \mathcal{W}_{2ct}$ be a subfamily of internally vertex-disjoint $(4ct,\ell+m^3, m^x,$ $\tfrac{t}{80}, m^3)$-webs as in Claim~\ref{many_web} and let $W=\bigcup_{i\in[2ct]}\mathsf{Ctr}(\mathcal{W}_i)\cap V_2$. Then $|W|\leq 8c^2t^2m^3$. We shall inductively construct a $K_{ct}$-subdivision as follows. Let $\mathcal{P}$ be a maximum collection of vertex-disjoint paths through $V_2$ of length at most $4m^3$ under the following rules.
\stepcounter{propcounter}
\begin{enumerate}[label = ({\bfseries \Alph{propcounter}\arabic{enumi}})]
\rm\item\label{con1} Each path $P_{ij}\in \mathcal{P}$ connects a pair of core vertices of unit from different webs $\mathcal{W}_i$ and $\mathcal{W}_j$. We say the corresponding units are then \textit{occupied}.

\rm\item\label{con2} For each pair of webs, there is at most one path in $\mathcal{P}$ between the core vertices from their respective units. All paths in $\mathcal{P}$ are internally disjoint from $W$.
\end{enumerate}
It is easy to see that the paths $P_{ij}$ can be extended using branches within $\mathcal{W}_i$ and $\mathcal{W}_j$ to a path of length at most $2\ell+8m^3$ joining the core vertices of those two webs.
Based on \ref{con2}, at most $4c^2t^2$ paths are built and thus $|V(\mathcal{P})|\le 16c^2t^2m^3$.
\begin{itemize}
    \item For each second-level branch in a unit, if more than $\tfrac{t}{80}$ (half of) leaves of the corresponding pendant star belong to $V(\mathcal{P})$ or a vertex in the second-level branch belongs to $V(\mathcal{P})$, then we say the second-level branch is \textit{used}.
     \item For each unit within a web, if more than $\tfrac{m^x}{2}$ second-level branches are used, then we say the unit \textit{over-used}.
      \item If more than $ct$ units within the web are over-used, then the web is \textit{bad}, and otherwise it is \textit{good}.
\end{itemize}
Note that all webs are internally vertex disjoint. As the paths in $\mathcal{P}$ do not intersect the centers of webs, a web is bad only when more than $ct$ units within the web are over-used, each such unit containing at least $\tfrac{m^x}{2}$ vertices in $V(\mathcal{P})$. This implies that each bad web contains at least $\tfrac{ctm^x}{2}$ vertices of $V(\mathcal{P})$ in its interior. Hence, there are at most $\tfrac{16c^2t^2m^3}{ctm^x/2}<ct$ bad webs. The following claim would give a desired $K_{ct}$-subdivision by combining $ct$ good webs and the paths in \ref{con1}.
\begin{claim}
    For every pair of good webs, there is a path in $\mathcal{P}$ between two core vertices from their respective units.
\end{claim}
\begin{proof}
Without loss of generality, we may assume for a contradiction that $\mathcal{W}_1$ and $\mathcal{W}_2$ are a pair of good units for which there is no desired path in $\mathcal{P}$.
If a web is good, then by definition it has at least $ct$ units that are not occupied and not over-used, namely, each such unit has at least $\tfrac{m^x}{2}$ second-level branches attached with at least $\tfrac{t}{80}$ leaves. Therefore, the new web obtained by deleting over-used units and all used second-level branches, still has an exterior of at least $\tfrac{ct^2m^x}{320}$ vertices. Hence, since $|W\cup V(\mathcal{P})|\leq t^2m^3$, by Lemma~\ref{random_distance}, we can connect the exteriors of two new webs by a path of length at most $m^3$ whilst avoiding $W\cup V(\mathcal{P})$. One can easily extend this path to a desired path as in \ref{con1}.
\end{proof}
\end{proof}
We now complement the proof of Subclaim~\ref{web-unit1} to end the section.
\begin{proof}[Proof of Subclaim~\ref{web-unit1}]
If we have $F_1,\ldots, F_{s}$ for some $0\leq s<2tm^{x+12}$, then the set $W:=\bigcup_{i\in[s]}V(F_i)$ has size at most $2tm^{x+12}[2m^x(m^3+\tfrac{t}{20})]\leq \tfrac{t^2m^{2x+12}}{4}$.
To build one more unit, we first claim that $G_2-W_2-W$ contains vertex-disjoint stars $S_1,\cdots, S_{tm^{4x}}$ where each $S_i$ has exactly $\tfrac{t}{30}$ leaves.
Indeed, we choose $G_2'\<G_2$ such that $\delta(G_2')\ge\tfrac{1}{2}d(G_2')\ge \tfrac{1}{2}d(G_2)$. Then by definition we obtain $C_{3\beta}(G_2')\ge C_{\beta}(G)$. Let $S_1,\cdots, S_{p}$ be a maximal collection of stars as required in $G_2'$ for some $0\leq p<tm^{4x}$. Then $W':=\bigcup_{i\in[p]}V(S_i)$ has size at most $\tfrac{t^2m^{4x}}{10}$.  Since  $G_2'$ is $(\tfrac{C_{\beta}(G)}{2},\tfrac{1}{10})$-dense and $|W_2\cup W'\cup W|\le t^2m^{2x}+\tfrac{t^2m^{4x}}{10}+\tfrac{t^2m^{2x+12}}{4}\leq t^2m^{4x}<\tfrac{C_{\beta}(G)}{3}$, Lemma~\ref{prop} \ref{robust} implies that $d(G_2'-W_2-W')\geq \tfrac{d(G'_2)}{10}\ge\tfrac{t}{30}$. Thus we can take a star denoted as $S_{p+1}$ with $\tfrac{t}{30}$ leaves, which is disjoint from $W'$, contrary to the maximality of $p$.

We take disjoint stars (renaming if necessary) $S_{1},\cdots, S_{tm^{3x-1}}$, $T_1,\ldots, T_{4tm^{4x-1}}$ as above, with centers $u_{1},\ldots, u_{tm^{3x-1}}$, $v_1,\ldots, v_{4tm^{4x-1}}$, respectively. Let $Z=\{u_{1},\ldots, u_{tm^{3x-1}},v_1,\ldots, v_{4tm^{4x-1}}\}$.
Let $\mathcal{P}$ be a maximum collection of internally disjoint paths $P_{ij}$ in $G_2-W-W_2$ satisfying the following rules.
\stepcounter{propcounter}
\begin{enumerate}[label = ({\bfseries \Alph{propcounter}\arabic{enumi}})]
\rm\item\label{1pathlab2} Each path $P_{ij}$ in $\mathcal{P}$ is a unique $u_i,v_j$-path of length at most $m^3$.
\rm\item\label{1pathlab3} Each $P_{ij}$ does not contain any vertex in $Z\cup W_2\cup W$ as an internal vertex.
\end{enumerate}

Now we claim that there is a center $u_i$ connected to at least $2m^x$ distinct centers $v_j$ via the paths in $\mathcal{P}$. Suppose to the contrary that every $u_i$ is connected to less than $2m^{x}$ centers $v_j$. Then $|\mathcal{P}|\leq 2tm^{4x-1}$ and $|V(\mathcal{P})|\leq 2tm^{4x-1} \cdot m^3\leq 2tm^{4x+2}$. Let
\begin{equation*}
U:=\bigcup_{i\in [tm^{3x-1}]}(S_i\backslash\{u_i\})\setminus V(\mathcal{P}),
\end{equation*}
and $V$ be the set of leaves of all stars $T_i$ whose centers are not used as endpoints of paths in $\mathcal{P}$. Then we have
\begin{equation}\label{webine1}
|U|\geq \tfrac{t^2m^{3x-1}}{30}-2t m^{4x+2}\geq\tfrac{t^2m^{3x-1}}{60}, \nonumber
\end{equation}
and
\begin{equation}\label{webine2}
|V|\geq \tfrac{t}{30}\cdot(4t m^{4x-1}-2tm^{4x-1})\geq \tfrac{t^2m^{4x-1}}{60}.\nonumber
\end{equation}
On the other hand, recall that $t\ge m^{3x}$, thus
\[
|W|+|W_2|+|\mathsf{Int}(\mathcal{P})|+|Z|
\leq  \tfrac{t^2m^{2x+12}}{4}+t^2m^{2x}+3t m^{5x+2}+4t m^{4x}
\leq t^2m^{2x+12}.\nonumber
\]
Hence, applying Lemma~\ref{random_distance} with $V_2,U,V,W\cup W_2\cup \mathsf{Int}(\mathcal{P})\cup Z$ playing the roles of $V,X_1,X_2,W$, respectively, we obtain vertices $x_{1}\in U, x_{2}\in V$ and a path through $V_2$ of length at most $m^3$ connecting $x_{1}$ and $x_{2}$ whilst avoiding vertices in $W\cup W_2\cup \mathsf{Int}(\mathcal{P})\cup Z$. Denote by $S_{k_1}, T_{k_2}$ the stars which contain $x_{1}, x_{2}$ as leaves, respectively. This yields a $u_{k_1},v_{k_2}$-path $P_{k_1,k_2}$, which is internally disjoint from $W\cup W_2\cup \mathsf{Int}(\mathcal{P})\cup Z$. Hence, $P_{k_1,k_2}$ satisfies \ref{1pathlab2} and \ref{1pathlab3}, a contradiction to the maximum of $\mathcal{P}$.

Therefore, there exists a center $u_i$ connected to $2m^{x}$ distinct centers $v_j$, say $v_1,\ldots v_{2m^{x}}$, which correspond to stars $T_1, \ldots T_{2m^{x}}$. Recall that all stars in $\{T_1, \ldots T_{2m^{x}}\}$ are vertex disjoint and the number of vertices in all $P_{i,j}$ $(j\in [2m^{x}])$ is at most $2m^{x+3}<\tfrac{t}{120}$. Hence, every $T_j$ $(j\in[2m^{x}])$ has at least $\tfrac{t}{40}$ leaves that are not used in $P_{i,j}$ for any $j\in [2m^{x}]$. These stars, together with the corresponding paths to $u_i$, form a desired unit in $G-W-W_2$. Thus, we can greedily pick vertex-disjoint units as above.
\end{proof}

\section{Concluding remarks}
%Lastly, we would like to further discuss when the logarithmic term in $\sqrt{\tfrac{C_{\alpha}(G)}{\log C_{\alpha}(G)}}$ could be removed.
%It is not even clear whether the ratio $C_{2\alpha}(G)/C_{\alpha}(G)$ is bounded by a constant term...

As discussed in the introduction, the term $\sqrt{\tfrac{C_{\alpha}(G)}{\log C_{\alpha}(G)}}$ is tight by considering the $d/1000$-blow-up of a $d$-vertex $1000$-regular expander, which consists of large complete bipartite graphs. It would be
possible that the logarithmic term  can be removed by excluding any fixed bipartite graph, thus generalising the result of Liu-Montgomery \cite{LiuC4}.

%\section*{Acknowledgements}
%The authors would  like to thank Xia Wang for helpful discussions.

\begin{appendix}

\section{Proof of Lemma~\ref{lem_sparse}}\label{im3}
As promised, we shall prove Lemma~\ref{lem_sparse}, and this improves a result in \cite{crux} as follows.
\begin{lemma}[\cite{crux}]\label{ebound}
    Suppose $\tfrac{1}{n}\ll c\ll \eps\ll \tfrac{1}{100}$ and $d\leq \exp(\log^{1/6}n)$. If an $n$-vertex $(\eps,\eps d)$-robust-expander $G$ has average degree at least $\tfrac{d}{2}$ and minimum degree $\delta(G)\geq \tfrac{d}{4}$, then $G$ contains a $K_{cd}$-subdivision.
\end{lemma}
We will use the following result due to Koml\'{o}s and Szemer\'{e}di to robustly find expanders.

\begin{lemma}[\cite{KS}]\label{stargre}
There exist $0<\varepsilon_0, \varepsilon_1<\tfrac{1}{8}$ such that for any $k\in \mathbb{N}$ every graph $G$ contains an $(\varepsilon_1,k)$-expander $H(V,E)$ with
\begin{equation*}
d(H)\geq \tfrac{d(G)}{1+\varepsilon_0}\geq\tfrac{d(G)}{2} \ \text{and} \ \delta(H)\geq \tfrac{d(H)}{2},
\end{equation*}
which has the following additional robust property where $n=|V|$. For every $X\subseteq V$ with $|X|<\tfrac{n\rho(n)d(H)}{4\Delta(H)}$, there is a subset $Y\subseteq V\setminus X$ of size $|Y|>n-\tfrac{2\Delta(H)}{d(H)}\cdot\tfrac{|X|}{\rho(n)}$ such that $H[Y]$ is still an $(\varepsilon_1,k)$-expander. Moreover, $d(H[Y])\geq \tfrac{d(H)}{2}$.
\end{lemma}

We will use a result in \cite[Lemma~4.2]{crux} which reduces the problem to finding clique subdivisions in expanders of bounded maximum degree.
\begin{lemma}[\cite{crux}]\label{maxd}
Suppose $\tfrac{1}{d}\ll c\ll \eps\ll1$ and $\eps d\leq k\leq d^2\log^9n$. Let $G$ be an $n$-vertex $(\eps,k)$-expander with $\delta(G)\geq \tfrac{d}{2}$. If $G$ does not contain a $K_{cd}$-subdivision, then $G$ contains an $(\tfrac{\eps}{2},k)$-expander subgraph $H$ with at least $\tfrac{n}{2}$ vertices such that $\delta(H)\geq \tfrac{d}{4}$ and $\Delta(H)\leq 10d^2\log^{10}n$.
\end{lemma}
The proof of Lemma~\ref{lem_sparse} essentially follows from that of Lemma~\ref{dense1}, and the main task therefore is to build many webs that have disjoint interiors. To do this, we first robustly grow a ball around a vertex as follows (in place of Lemma~\ref{ball}).
\begin{lemma}\label{ball_1}
Suppose $\tfrac{1}{n},\tfrac{1}{d}\ll c \ll \eps\le\tfrac{1}{5}$ and $d\le \tfrac{\sqrt{n}}{2\log^{20}n}$. Let $G$ be an $n$-vertex  $(\eps,\eps d)$-expander with $\delta(G)\geq \tfrac{d}{100}$. Let  $\ell=\log^4(d\log^{40}n)$ and  $P_1,\dots, P_t$ be consecutive shortest paths from a fixed vertex $v$ in $B_{G}^{\ell}(v)$. Writing $U=\bigcup_{i\in t}V(P_i)$, if  $t\leq cd$, then $|B_{G-(U\backslash\{v\})}^{\ell}(v)|\geq d^2\log^{40}n$.
\end{lemma}
Lemma~\ref{ball_1} follows from the same proof as in \cite[Lemma~4.8]{GeneralH} using slightly different parameters.

\begin{proof}[Proof of Lemma~\ref{lem_sparse}]
We may assume that $d\geq \log^{50}n$ in the whole process of the proof, otherwise Lemma \ref{ebound} implies that the result holds. Choose $\tfrac{1}{d}\ll c \ll b\ll \eps\ll 1$, and $d,n$ satisfy $\log^{50}n \leq d \leq \sqrt[3]{\tfrac{n}{\log^{100}n}}$. Let $G$ be an $n$-vertex $(\eps, \eps d)$-expander with average degree at least $\tfrac{d}{2}$ and minimum degree $\delta(G) \geq \tfrac{d}{4}$.
Suppose to the contrary that $G$ contains no $K_{cd}$-subdivision, then by Lemma~\ref{maxd} with $(k,d)=(\eps d,\tfrac{d}{2})$, there is an $(\tfrac{\eps}{2},\eps d)$-expander $H$ in $G$ such that $|H|\geq \tfrac{n}{2}$, $\delta(H)\geq \tfrac{d}{8}$ and $\Delta(H)\leq 3d^2\log^{10}n$. In the following, we shall find a $K_{cd}$-subdivision in $H$ to yield a contradiction. Let $m=\log^4\tfrac{n}{d}$.
\begin{claim}\label{lem_d3_web}
   For any set $W\subseteq V(H)$ with $|W|\leq d^2m^9$, $H-W$ contains a $(4cd,4m,m^8,\tfrac{d}{40},m+2)$-web.
\end{claim}
\begin{proof}[Proof of Claim \ref{lem_d3_web}]
We use Lemma~\ref{stargre} to greedily find vertex-disjoint stars.
\begin{itemize}
    \item[$(1)$] For any set $X$ of size at most $d^2m^{10}$, the graph $H-W-X$ contains vertex-disjoint stars $S_1,\cdots, S_{2dm^{24}}$ where each $S_i$ has exactly $\tfrac{d}{10}$ leaves.
\end{itemize}

Indeed, if we have $S_1,\cdots, S_\ell$ for some $0\leq \ell<2dm^{24}$, then the set $W_1:=\bigcup_{i\in[\ell]}V(S_i)$ has size at most $d^2m^{24}$. By Lemma~\ref{stargre} and the assumption $d \leq \sqrt[3]{\tfrac{n}{\log^{100}n}}$, there exists $Y\subseteq V(H)\setminus (W_1\cup W\cup X)$ such that \[|Y|\geq \tfrac{n}{2}-\tfrac{2\Delta(H)|W_1\cup W\cup X|}{d(H)\rho(n)}\geq \tfrac{n}{4}, ~\text{and}~ d(H[Y])\geq \tfrac{d}{8}.\]  Thus $H-W-X$ contains a star $S_{\ell+1}$ with $\tfrac{d}{10}$ leaves disjoint from $W_1$. By repeating this for $s=0,1,\ldots,2dm^{24}$, we obtain the desired stars.

\begin{itemize}
    \item[$(2)$] The graph $H-W$ contains vertex-disjoint $(2m^8,\tfrac{d}{20},m+2)$-units $F_1,\ldots, F_{2dm^2}$.
\end{itemize}

If we have $F_1,\ldots, F_{\ell}$ for some $0\leq \ell<2dm^2$, then the set $W_2:=\bigcup_{i\in[\ell]}V(F_i)$ has size at most $2dm^2[2m^8(m+2+\tfrac{d}{20})]\leq \tfrac{d^2m^{10}}{4}$. By $(1)$, $H-W-W_2$ contains vertex-disjoint stars $S_{1},\cdots, S_{dm^{12}}$, $T_1,\ldots, T_{dm^{24}}$ with centers $u_{1},\ldots, u_{dm^{12}}$, $v_1,\ldots, v_{dm^{24}}$, respectively, where each of these stars has exactly $\tfrac{d}{10}$ leaves. We denote $Z=\{u_{1},\ldots, u_{dm^{12}},v_1,\ldots, v_{dm^{24}}\}$.

Let $\mathcal{P}$ be a maximum collection of internally disjoint paths $P_{ij}$ in $H-W-W_2$ satisfying the following rules.
\stepcounter{propcounter}
\begin{enumerate}[label = ({\bfseries \Alph{propcounter}\arabic{enumi}})]
\rm\item\label{1pathlab2_1} Each path $P_{ij}$ in $\mathcal{P}$ is a unique $u_i,v_j$-path of length at most $m+2$.
\rm\item\label{1pathlab3_1} Each $P_{ij}$ does not contain any vertex in $Z\cup W_2\cup W$ as an internal vertex.
\end{enumerate}

Now we declare that there is a center $u_i$ connected to at least $2m^8$ distinct centers $v_j$ via the paths in $\mathcal{P}$. Suppose to the contrary that every $u_i$ is connected to less than $2m^{8}$ centers $v_j$. Then $|\mathcal{P}|\leq 2dm^{20}$ and $|V(\mathcal{P})|\leq 2dm^{20} \cdot(m+2)\leq3dm^{21}$. Let
\begin{equation*}
U:=\left(\bigcup_{i\in [dm^{12}]}(V(S_i)\backslash\{u_i\})\right)\setminus V(\mathcal{P}),
\end{equation*}
and $V$ be the set of leaves of all stars $T_i$ whose centers are not used as endpoints of paths in $\mathcal{P}$. Then we have
\begin{equation}\label{webine1_1}
|U|\geq \tfrac{d^2m^{12}}{10}-3d m^{21}\geq\tfrac{d^2m^{12}}{20},
\end{equation}
and
\begin{equation}\label{webine2_1}
|V|\geq \tfrac{d}{10}\cdot(d m^{24}-2dm^{20})\geq \tfrac{d^2m^{12}}{20}.
\end{equation}
On the other hand,
\begin{equation}
|W|+|W_2|+|\mathsf{Int}(\mathcal{P})|+|Z|
\leq d^2m^9+ \tfrac{d^2m^{10}}{4}+3dm^{21}+dm^{12}+dm^{24}
\leq d^2m^{10}.
\end{equation}
Thus, applying Lemma~\ref{distance} with $U,V$, $W\cup W_2\cup \mathsf{Int}(\mathcal{P})\cup Z$ playing the roles of $X_1,X_2,W$, respectively, we obtain vertices $x_{1}\in U, x_{2}\in V$ and a path of length at most $m$ connecting $x_{1}$ and $x_{2}$ whilst avoiding vertices in $W\cup W_2\cup \mathsf{Int}(\mathcal{P})\cup Z$. Denote by $S_{k_1}, T_{k_2}$ the stars which contain $x_{1}, x_{2}$ as leaves, respectively. This yields a $u_{k_1},v_{k_2}$-path $P_{k_1,k_2}$, which is internally disjoint from $W\cup W_2\cup \mathsf{Int}(\mathcal{P})\cup Z$. Hence, $P_{k_1,k_2}$ satisfies \ref{1pathlab2} and \ref{1pathlab3}, a contradiction to the maximum of $\mathcal{P}$.

Therefore, there exists a center $u_i$ connected to $2m^{8}$ distinct centers $v_j$, say $v_1,\ldots v_{2m^{8}}$, which correspond to stars $T_1, \ldots T_{2m^{8}}$. Recall that all stars in $\{T_1, \ldots T_{2m^{8}}\}$ are vertex-disjoint and the number of vertices in all $P_{i,j}$ $(j\in [2m^{8}])$ is at most $2m^{8}(m+2)<\tfrac{d}{20}$. Hence, every $T_j$ $(j\in[2m^{8}])$ has at least $\tfrac{d}{20}$ leaves that are not used in $P_{i,j}$ for any $j\in [2m^{8}]$. These stars, together with the corresponding paths to $u_i$, form a desired unit in $H-W-W_2$. Thus, we can greedily pick vertex-disjoint units as above.

Now we get pairwise vertex-disjoint $(2m^{8},\tfrac{d}{20},m+2)$-units $F_{1},\ldots,F_{2dm^2}$, and denote by $f_i$ the core vertex of $F_i$ and $F=\{f_{1},\ldots,f_{2dm^2}\}$. Let $W_3=\bigcup_{i\in [2dm^2]}V(F_i)$ and $W_4=\bigcup_{i\in [2dm^2]}\mathsf{Int}(F_i)$. Then we have $|W_4|\leq 4dm^{11}$ and  $|W_3|\leq \tfrac{d^2m^{10}}{4}$. As $|W|+|W_3|\leq d^2m^9+\tfrac{d^2m^{10}}{4}<d^2m^{10}$, by Lemma~\ref{stargre}, we again get that there exists $Y'\subseteq V(H)\setminus (W\cup W_3)$ with $|Y'|\geq \tfrac{n}{2}$ such that the subgraph $G_1:=H[Y']$ is still an $(\tfrac{\eps}{2}, \eps d)$-expander with $\delta(G_1)\ge \tfrac{1}{2}d(G_1)\ge \tfrac{d}{16}$ and pick a vertex $v\in V(G_1)$. We claim that there is a $(4cd,4m,m^8,\tfrac{d}{40},m+2)$-web centered at $v$.

To build the desired $(4cd,4m,m^8,\tfrac{d}{40},m+2)$-web, we shall proceed by finding $4cd$ internally vertex-disjoint paths $Q_1,\ldots,Q_{4cd}$ in $G$ from $v$ satisfying the following rules, where $\ell$ is a constant obtained from Lemma~\ref{ball_1} applied on $G_1$.
\stepcounter{propcounter}
\begin{enumerate}[label = ({\bfseries \Alph{propcounter}\arabic{enumi}})]
\rm\item\label{ballab1} Each path $Q_i$ is a unique $v,f_{x_i}$-path of length at most $3m+\log^4(d\log^{40}n)\leq 4m$, where $x_i\in[2dm^2]$.
\rm\item\label{ballab2} Each path does not contain any vertex in $W_4\cup W$ as an internal vertex.
\rm\item\label{ballab3} The sequence of subpaths $Q_i[B^{\ell}_{G_1}(v)]$ $(i\in[4cd])$ form consecutive shortest paths from $v$ in $B^{\ell}_{G_1}(v)$.
\end{enumerate}
Assume that we have iteratively obtained a collection of shortest paths $\mathcal{Q}=\{Q_1,\ldots,Q_s\}$ $(0\leq s<4cd)$ as in \ref{ballab1}-\ref{ballab3}. Then $|\mathsf{Int}(\mathcal{Q})|<16cdm$. Note that \ref{ballab3} gives $s$ consecutive shortest paths $P_1,\ldots,P_s$ from $v$ in $B_{G_1}(v)$, where $P_i=Q_i[B^{\ell}_{G_1}(v)]$ and write $\mathcal{P}=\{P_1,\ldots,P_s\}$. Applying Lemma~\ref{ball_1} to $G_1$, we get
$$|B^{\ell}_{G_1-\mathsf{Int}(\mathcal{Q})}(v)|=|B_{G_1-\mathsf{Int}(\mathcal{P})}^{\ell}(v)|\geq d^2\log^{40}(\tfrac{n}{2})\geq d^2m^{10}.$$
We call a unit $F_i$ \textit{available} if its core vertex is not used as an endpoint of a path in $\mathcal{Q}$.
Let $U'$ be the leaves of all pendent stars in available units. Then we have
$$|U'|\geq 2m^{8}\cdot\tfrac{d}{20}(2dm^2-4cd)>\tfrac{d^2m^{10}}{10}.$$
Note that
$$|W|+|V(\mathcal{Q})|+|W_4|\leq d^2m^9+16cdm+4dm^{11}<2d^2m^9. $$
Applying Lemma~\ref{distance} on $H$ with $B_{G_1-\mathsf{Int}(\mathcal{Q})}^{\ell}(v)$, $U'$, $W_4\cup W\cup V(\mathcal{Q})$ playing the roles of $X_1,X_2,W$, respectively, we can find a shortest path, say $Q$, joining a vertex $w$ from $B_{G_1-\mathsf{Int}(\mathcal{Q})}^{\ell}(v)$ to a leaf $u_j\in \mathsf{Ext}(F_j)$ for some $j\in [2dm^2]$. One can easily find a $v,w$-path within $B^{\ell}_{G_1-\mathsf{Int}(\mathcal{Q})}(v)$, denoted as $P_{s+1}$,  and a $u_j,f_j$-path, denoted as $R_{s+1}$, inside the unit $F_j$. Let $Q_{s+1}=P_{s+1}QR_{s+1}$. Then the paths $Q_1,\ldots,Q_{s+1}$ satisfy \ref{ballab1}-\ref{ballab3}. Repeating this for $s=0,1,\ldots,4cd$, yields $4cd$ paths $Q_1,\ldots, Q_{4cd}$ as desired.

For every $i\in [2dm^2]$, we say a pendent star in $F_i$ is \emph{overused} if at least $\tfrac{d}{40}$ leaves are used in $V(\mathcal{Q})$ and a unit is \emph{bad} if at least $m^8$ pendent stars are overused. Then there are at most $\tfrac{16cdm}{d/40}=160cm<m^{8}$ overused pendent stars in total, and therefore there is no bad unit. By removing from every unit $F_i$ the branches attached with overused pendant stars and combining corresponding paths $Q_i$,  we obtain a $(4cd,4m,m^8,\tfrac{d}{40},m+2)$-web as desired.
\end{proof}
By repeatedly using Claim \ref{lem_d3_web}, we can greedily find $d$ distinct $(4cd,4m,m^8,\tfrac{d}{40},m+2)$-webs in $H$ that have disjoint interiors, as the total number of internal vertices of these webs is at most
\begin{equation*}
     d\cdot 4cd\cdot (4m+(m+2)\cdot m^8)\leq 20cd^2m^9<d^2m^9.
\end{equation*}
By the same connection process as in the proof of Lemma~\ref{webs} (see \ref{con1}-\ref{con2} in \textbf{Phase} $(2)$), there is a $K_{cd}$-subdivision in $H$ as promised.
\end{proof}

\section{Proof of Lemma~\ref{random_distance}}\label{sec4.1}
The proof of Lemma~\ref{random_distance} is immediately derived from the following result.
\begin{lemma}\label{randomball}
Suppose $\tfrac{1}{n},\tfrac{1}{d}\ll \eps <\tfrac{1}{5}$, $d\ge \log^{100} n$ and $x,k\in \mathbb{N}$ with $x\geq 11$. %\red{$k\geq \tfrac{n}{\log^pn}$ for $p\geq 1000$}\todo{check if need this}.
    Let $G$ be an $n$-vertex $(\eps,k)$-robust-expander with average degree $d$ and $V\subseteq V(G)$ be a random subset chosen by including each vertex in dependently at random with probability $\tfrac{1}{2}$. Let $m=\log^4\tfrac{n}{k}$. Then with probability at least $1-o(\tfrac{1}{n})$, for every $U,W\subseteq V(G)$ with $km^9\le |U|\leq \tfrac{2n}{3}, |W|= \tfrac{|U|}{m^x}$
    \begin{equation*}
        |B^{m^2}_{G-W}(U)|>\tfrac{|V|}{2}.
    \end{equation*}
\end{lemma}
We shall first introduce some more notation and tools for the proof.
Let $N_{G,\lambda}(U):=\{v\in V(G)\setminus U: |N_G(v)\cap U|\geq \lambda\}$, that is, the set of vertices in $V(G)\setminus U$ of degree at least $\lambda$ towards $U$. The following lemma shows that for every vertex subset $U$, either $U$ or its robust neighborhood of vertices with plenty of edges towards $U$ expands nicely.
\begin{lemma}\label{prop12}
Suppose $\tfrac{1}{n}\ll \eps <\tfrac{1}{5}$ and let $G$ be an $n$-vertex $(\eps,k)$-expander with average degree $d$ and $m=\log^4\tfrac{n}{k}$, and let $U\subseteq V(G)$ with $\tfrac{k}{2}\le |U|\leq \tfrac{n}{2}$.
  Then for any $0<\lambda<d$, either
    \begin{equation*}
        \mathbf{(a)}. \ |N_{G}(U)|\geq \tfrac{d|U|}{\lambda m}, \ \ \ \ \ \text{or} \ \ \ \ \ \mathbf{(b)}. \ |N_{G,\lambda}(U)|\geq \tfrac{|U|}{m}.
    \end{equation*}
\end{lemma}

\begin{proof}
    Suppose that $\mathbf{(a)}$ is not true, that is, $|N_{G}(U)|<\tfrac{d|U|}{\lambda m}$. Let $X=N_{G}(U)\setminus N_{G,\lambda}(U)$, so that $|X|\leq |N_{G}(U)|<\tfrac{d|U|}{\lambda m}$.  Let $F$ be the edges of $G$ between $U$ and $X$, so that
\[
        |F|<\lambda|X|<\tfrac{d|U|}{m}<\rho(|U|)|U|d.
\]
    Then we obtain that $\mathbf{(b)}$ holds as
\[
        |N_{G,\lambda}(U)|\ge |N_{G-F}(U)|\geq \rho(|U|)|U|\ge \tfrac{|U|}{m}.
\]
\end{proof}
The following result was first observed in \cite{Mon-cycle}, and we make a slight adaptation for our proof.
\begin{lemma}\label{prop13}
Suppose $\tfrac{1}{n}\ll \eps <\tfrac{1}{5}$ and $x,k\in \mathbb{N}$ with $x\geq 2$.  Let $G$ be an $n$-vertex $(\eps,k)$-robust-expander with average degree $d$, $m=\log^4\tfrac{n}{k}$ and $U, W\subseteq V(G)$ with $k\leq|U|\leq \tfrac{2n}{3}, |W|\leq \tfrac{|U|}{m^x}$. Then for all positive integers $\lambda,s,t$ with $s\ge 8m, t\ge 2\lambda$ and $d\ge 10tm\lambda$, we can find in $G-W$ either
    \begin{itemize}
        \item[$(1)$] $\tfrac{|U|}{s}$ vertex-disjoint stars, whose centers in $U$ and each with $t$ leaves such that all lie in $V(G)\setminus (U\cup W)$, or
        \item[$(2)$] a bipartite subgraph $H$ with vertex classes $U$ and $X\subseteq V(G)\setminus (U\cup W)$ such that $|X|\geq \tfrac{|U|}{4m}$ and every vertex in $X$ has degree at least $\lambda$ in $H$ and every vertex of $U$ has degree at most $2t$ in $H$.
    \end{itemize}
\end{lemma}
\begin{proof}
    First, we take a maximal collection of vertex-disjoint $t$-stars in $G-W$ such that all centers lie in $U$ and leaves outside $U$. Denote by $C,L$ the set of centers and the set of all leaves, respectively, of these stars. If $(1)$ dose not hold, then we can assume that $|C|< \tfrac{|U|}{s}$ and $|L|<\tfrac{|U|t}{s}$. By the maximality, every vertex in $U\setminus C$ has less than $t$ neighbors in $V(G-W)\backslash(U\cup L)$. Thus,
    \begin{equation}\label{UCupp}
        |N_{G-W}(U\setminus C)|\leq |C|+|L|+(|U|-|C|)t\leq 2|U|t.
    \end{equation}
    Second, we aim to construct the bipartite subgraph $H$ through the following process. Let $X_0=\emptyset$, and set $H_0$ to be the graph with vertex set $U\cup X_0$ and no edges. Let $r=|V(G)\setminus (U\cup W)|$ and label the vertices of $V(G)\setminus (U\cup W)$ arbitrarily $v_1,v_2,\ldots, v_r$. For each $i\geq 1$, if possible pick a star $S_i$ in $G-W$ with center $v_i$ and $\lambda$ leaves in $U$ such that $H_{i-1}\cup S_i$ has maximum degree at most $2t$, and let $H_i=H_{i-1}\cup S_i$ and $X_i=X_{i-1}\cup \{v_i\}$, while otherwise we set $H_i=H_{i-1}$ and $X_i=X_{i-1}$. Finally, we shall claim that
$(2)$ holds for $H_r$ with bipartition $(U,X_r)$.

By the construction, $\Delta(H_i)\leq 2t$ holds for each $i\in[r]$, and every vertex $v_i$ in $X_r$ has degree exactly $\lambda$ in $H_r$. It remains to show that $|X_r|\geq \tfrac{|U|}{4m}$. Suppose for contradiction that $|X_r|< \tfrac{|U|}{4m}$  and let $U'=\{v\in U\setminus C: d_{H_r}(v)=2t\}$. Due to the maximality of the family of vertex-disjoint stars defining $C$ and $L$ we picked before, each vertex in $U'$ has fewer $t$ neighbors in $X_r\backslash(L\cup W)$. That is, it must have at least $t$ neighbors in $X_r\cap L$. Recall that for each $v\in X_r\cap L$, we have $d_{H_r}(v)= \lambda$. Thus,
\begin{equation*}
    |U'|\leq \tfrac{\lambda|X_r\cap L|}{t}\leq \tfrac{\lambda|X_r|}{t}\leq \tfrac{\lambda |U|}{4tm}.
\end{equation*}
Let $B=C\cup U'$. Then $|B|\leq \tfrac{|U|}{s}+ \tfrac{\lambda |U|}{4tm}\leq \tfrac{|U|}{4m}$, and thus $|U\setminus B|\geq \tfrac{2|U|}{3}\geq \tfrac{k}{2}$. Applying Lemma \ref{prop12} to $U\setminus B$, we have either
\begin{equation*}
    |N_{G-W}(U\setminus B)|\geq \tfrac{d|U\setminus B|}{\lambda m}-|W|\geq \tfrac{d|U\setminus B|}{2\lambda m}, \ \ \text{or} \ \ |N_{G-W,\lambda}(U\setminus B)|\geq \tfrac{|U|}{m}-|W|\geq \tfrac{|U|}{2m}.
\end{equation*}
Since
\begin{equation*}
    |N_{G-W}(U\setminus C)|\geq |N_{G-W}(U\setminus B)|-|U|\geq \tfrac{d|U\setminus B|}{2\lambda m}-|U|\geq \tfrac{d|U|}{3\lambda m}-|U|>2|U|t,
\end{equation*}
a contradiction to (\ref{UCupp}). Therefore, $|N_{G-W,\lambda}(U\setminus B)|\geq   \tfrac{|U|}{2m}>|B|+|X_r|$. Note that every vertex $v$ in $N_{G-W,\lambda}(U\setminus B)\setminus (B\cup X_r)$ has at least $\lambda$ neighbors in $U\setminus B$ in $G-W$ while every vertex in $U\setminus B$ has degree strictly less than $2t$ in $H_r$. Thus we can further add a $\lambda$-star centered at $v$ to $H_r$, a contradiction to the maximality of $H_r$.
\end{proof}

Lemma \ref{prop18} depicts that any vertex $U$ in an $(\eps,k)$-expander $G$ contains a ``not small'' subset $U'$ such that $U'$ also expands well, and so that $U'$ captures lots of the expansion properties of $U$.
\begin{lemma}\label{prop18}
    Suppose $\tfrac{1}{n},\tfrac{1}{d}\ll \eps <\tfrac{1}{5}$ and $k\in \mathbb{N}$. Let $G$ be an $n$-vertex $(\eps,k)$-robust-expander with average degree $d$ and $U\subseteq V(G)$ with $\tfrac{k}{2}
    \le|U|\leq \tfrac{2n}{3}$. Let $m=\log^4\tfrac{n}{k}$. Then for any positive constant $\kappa\le \tfrac{d}{m}$, there exists $U'\subseteq U$ with $|N_{G}(U')|\geq \kappa|U'|$ and $|U'|> \tfrac{|U|}{\kappa m}$.
\end{lemma}
\begin{proof}
    We may assume that $|U|\geq 1$, otherwise we just take $U'=\emptyset$ satisfying the required conditions. Let $U'\subseteq U$ be a maximal set subject to $|N_{G}(U')|\geq \kappa|U'|$, and this is possible as $U'=\emptyset$ satisfies these conditions. If $U=U'$, then $U$ satisfies the condition itself. Next, we assume that $U\neq U'$. By the maximality of $U'$, we have that $|N_{G}(U')|<\kappa(|U'|+1)$, and every vertex $u$ in $U\setminus U'$ satisfying that $d_{G-N_G[U']}(u)< \kappa$.  Let $F$ be the set consisting of edges $uv$ with $u\in U\setminus U'$ and $v\in V(G)-N_G[U']$. Then $|F|<\kappa |U\setminus U'|<\tfrac{|U|d}{m}<\rho(|U|)|U|d$. As $G$ is an $(\eps,k)$-robust-expander, we obtain
    \begin{equation*}
        \rho(|U|)|U|\leq |N_{G-F}(U)|\leq |N_{G}(U')|<\kappa(|U'|+1).
    \end{equation*}
    Hence, $|U'|\geq \tfrac{\rho(|U|)|U|}{2\kappa}>\tfrac{|U|}{\kappa m}$.
\end{proof}
We also introduce the well-known martingale concentration result, which will be used to prove Lemma \ref{randomball}.
\begin{lemma}[\cite{Alon}]\label{Lip}
    Suppose that $X:\prod_{i=1}^n\Omega_i\rightarrow \mathbb{R}$ is $k$-Lipschitz. Then, for each $t>0$,
    \begin{equation*}
        \mathbb{P}(|X-\mathbb{E}X|>t)\leq 2\exp\left(\tfrac{-2t^2}{k^2N}\right).
    \end{equation*}
\end{lemma}
\begin{lemma}\label{randomball1}
Suppose $\tfrac{1}{n},\tfrac{1}{d}\ll \eps <\tfrac{1}{5}$ and $x,k,\kappa\in \mathbb{N}$ with $x\geq 2, d\geq \log^{100}n$.
Let $G$ be an $n$-vertex $(\eps,k)$-expander with average degree $d$ and $m=\log^4\tfrac{n}{k}$ and choose $\kappa\ge m^8$. Given subsets $U,W\subseteq V(G)$ with $|N_{G}(U)|\geq \kappa |U|$ for $k\leq|U|\leq \tfrac{2n}{3}$ and $|W|\leq \tfrac{|U|}{m^x}$, and let $V$ be a random subset chosen by including each vertex independently at random with probability $\tfrac{1}{2}$. Then, with probability $1-\exp\left(-\Omega(\tfrac{\kappa|U|}{m^{6}})\right)$,
    \begin{equation*}
        |B_{G-W}^{m^2}(U,V)|>\tfrac{|V|}{2}.
    \end{equation*}
\end{lemma}
\begin{proof}
Take $q=\frac{9}{20}, \ell=m^2$ and $p$ be such that $1-(1-p)^{\ell-1}(1-q)=\frac{1}{2}$. Then $(1-p)^{m^2-1}=\tfrac{10}{11}$ and thus $p\geq \tfrac{1}{11m^2}$.
For each $i\in[m^2]$, let $V_i$ be a random subset of $V(G)$ with each vertex included independently at random with probability $p$ if $i\leq m^2-1$ and with probability $q$ if $i=m^2$. Set $V=V_1\cup \ldots \cup V_{m^2}$. Recall that each vertex is included in $V$ independently at random with probability $\tfrac{1}{2}$. For each $0\leq i\leq m^2$, let $B_i$ be the set of vertrices of $G$ which can be reached via a path in $G-W$ which starts in $U$ and has length at most $i$ and whose internal vertices (if there are any) are in $V_1\cup \ldots \cup V_{i-1}$. In particular, let $B_0=U$ and $B_1=B_{G-W}(U)$. Note that $B_0\subseteq B_1\subseteq \ldots \subseteq B_{m^2}$.
 By the construction of $B_i$, we know that it is completely determined by the sets $U,V_1,\ldots, V_{i-1}$ while independent of $V_i$, and any vertex in $N_{G-W}(B_i)$ with a neighbor in $B_i$ that gets sampled into $V_i$ belongs to $B_{i+1}$. We aim to evaluate the value of $|B_{m^2}\cap V_{m^2}|$ in order to get the lower bound of $|B_{G-W}^{m^2}(U,V)|$.

Now we have the following claim, which states that either $B_i$ is very large or $B_{i+1}$ is likely to be larger than $B_i$. Denote by $N_G(v,A):=N_G(v)\cap A$ for some $A\subseteq V(G)$.
    \begin{claim}\label{largep}
        For each $1\leq i\leq m^2$, with probability $1-\exp\left(-\Omega(\tfrac{\kappa|U|}{m^{6}})\right)$, either $|B_i|\geq \tfrac{2n}{3}$ or $|B_{i+1}\setminus B_i|\geq \tfrac{|B_i|}{128m}$.
    \end{claim}
    \begin{proof}[Proof of Claim]
       Observe that
        \begin{equation*}
            \{v\in N_{G-W}(B_i): N_{G-W}(v,B_i)\cap V_i\neq \emptyset\}\subseteq B_{i+1}\setminus B_i.
        \end{equation*}
        We shall show that, for any set $Z\subseteq V(G)$ with $|Z|\leq \tfrac{2n}{3}$ and $B_1\subseteq Z$, the following holds.
        \begin{equation}\label{hpro}
            \mathbb{P}\left(|\{v\in N_{G-W}(Z):N_{G-W}(v,Z)\cap V_i\neq \emptyset\}|\geq \tfrac{|Z|}{128m}\right)\geq 1-\exp\left(-\Omega(\tfrac{|B_1|}{m^{6}})\right).
        \end{equation}
        Hence, for all $1\leq i\leq m^2-1$
        \begin{align}
         &\mathbb{P}\left(|B_i|\geq \tfrac{2n}{3} \ \ \text{or} \ \ |B_{i+1}\setminus B_i|\geq \tfrac{|B_i|}{128m}\right) \geq \mathbb{P}\left(|B_{i+1}\setminus B_i|\geq \tfrac{|B_i|}{128m} \mid |B_i|\leq \tfrac{2n}{3} \right) \nonumber \\
         &\geq \mathbb{P}\left(|\{v\in N_{G-W}(B_i): N_{G-W}(v,B_i)\cap V_i\neq \emptyset\}|\geq \tfrac{|B_i|}{128m} \mid |B_i|\leq \tfrac{2n}{3}\right) \nonumber \\
         &\geq 1-\exp\left(-\Omega(\tfrac{|B_1|}{m^{6}})\right)\geq 1-\exp\left(-\Omega(\tfrac{\kappa|U|}{m^{6}})\right), \nonumber
        \end{align}
       where the last second inequality we used $|B_1|=|B_{G-W}(U)|\geq \kappa|U|-|W|\ge \tfrac{\kappa|U|}{2}$. Let $Z\subseteq V(G)$ with $|Z|\leq \tfrac{2n}{3}$ and $B_1\subseteq Z$. As $|Z|\leq \tfrac{2n}{3}$ and $|W|\leq \tfrac{|U|}{m^x}\leq \tfrac{|B_1|}{m^x}\leq \tfrac{|Z|}{m^x}$, we can apply Lemma \ref{prop13} to $Z$ and $W$ with
       \[\lambda=m^2,~s=8m,
    ~t=8m^2,\]
    to show that (\ref{hpro}) holds in either case as follows.

       Suppose $G-W$ contains $\tfrac{|Z|}{8m}$ vertex-disjoint stars with $8m^2$ leaves, with the center in $Z$ and all leaves $V(G)\setminus Z$. Let $C\subseteq Z$ be the set of centers of these stars and $|C|=\tfrac{|Z|}{8m}$. Note that
       \begin{equation}\label{chn}
           |\{v\in N_{G-W}(Z):N_{G-W}(v,Z)\cap V_i\neq \emptyset\}|\geq 8m^2|C\cap V_i|.
       \end{equation}
       By the Chernoff bound and $p\geq \tfrac{1}{11m^2}$, with probability at least $1-\exp\left(-\tfrac{p|C|}{8}\right)=1-\exp\left(\Omega(\tfrac{|Z|}{m^3})\right)$, we have $|C_i\cap V_i|\geq \tfrac{p|C|}{2}\geq \tfrac{|Z|}{240m^3}$. Combining this with (\ref{chn}) we have that (\ref{hpro}) holds.

    Suppose that there is a bipartite subgraph $H\subseteq G-W$ with partition $Z$ and $X$ such that $|X|\geq \tfrac{|Z|}{4m}$ and every vertex in $X$ has degree at least $m^2$ in $H$ and every vertex of $Z$ has degree at most $16m^2$ in $H$. For each $v\in  X$, the probability that $v$ has no neighbors in $H$ in $V_i$ is at most $(1-p)^{m^2}\leq e^{-pm^2}\leq e^{-1/11}\leq \tfrac{15}{16}$. Let $Y$ be the random variable counting the number of vertices of $X$ having a neighbor in $V_i$ in $H$, so that $\mathbb{E}[Y]\geq \tfrac{|X|}{16}$. Observe that $Y$ is $16m^2$-Lipschitz since for each $v\in Z$, the event $\{v\in V_i\}$ affects $Y$ by at most $d_H(v)\leq 16m^2$. Hence by Lemma \ref{Lip}
    \begin{equation*}
        \mathbb{P}\left(Y<\tfrac{|X|}{32}\right)\leq \mathbb{P}\left(Y<\mathbb{E}[Y]-\tfrac{|X|}{32}\right)\leq 2\exp\left(-\tfrac{|X|^2}{2^{11}m^{4}|Z|}\right)\leq \exp\left(-\Omega(\tfrac{|Z|}{m^{6}})\right).
    \end{equation*}
    Each vertex in $X$ with a neighbor in $V_i$ in $H$ lies in $\{v\in N_{G-W}(Z):N_{G-W}(v,Z)\cap V_i\neq \emptyset\}$, so with probability $\exp\left(-\Omega(\tfrac{|Z|}{m^{6}})\right)$, we have $|\{v\in N_{G-W}(Z):N_{G-W}(v,Z)\cap V_i\neq \emptyset\}|\geq Y\geq \tfrac{|X|}{32}\geq \tfrac{|Z|}{128m}$, and thus (\ref{hpro}) holds.
    \end{proof}
    As $B_{m^2}$ and $V_{m^2}$ are independent and $\frac{2}{3}q=\frac{3}{10}>\frac{4}{15}$, by Chernoff's bound, we have that
    \begin{equation*}
        \mathbb{P}\left(|B_{m^2}\cap V_{m^2}|\leq \frac{4}{15}n \mid |B_{m^2}|\geq \frac{2}{3}n\right)\leq \mathbb{P}\left( \mathsf{Bin} \big(\frac{2}{3}n,q\big)\le \frac{4}{15}n \right)\le e^{-\Theta(n)}.
    \end{equation*}
    As $\mathbb{P}\left(|V|\geq \frac{8}{15}n \right)\leq e^{-\Theta(n)}$.
    In conclusion, we have that
    \begin{itemize}
        \item[$(i)$] for each $i\in [m^2-1]$, $|B_i|\geq \tfrac{2n}{3}$ or $|B_{i+1}\setminus B_i|\geq \tfrac{|B_i|}{128m}$, and
        \item[$(ii)$] $|B_{m^2}|<\tfrac{2n}{3}$ or $|B_{m^2}\cap V_{m^2}|> \frac{4}{15}n$, and
        \item[$(iii)$] $|V|\leq \frac{8}{15}n$
    \end{itemize}
    with probability at least
    \begin{equation*}
        1-m^2\exp\left(-\Omega(\tfrac{\kappa|U|}{m^{6}})\right)-e^{-\Theta(n)}\geq 1-\exp\left(-\Omega(\tfrac{\kappa|U|}{m^{6}})\right).
    \end{equation*}
    On the other hand, if $(i)-(iii)$ all hold, then for each $i\in[m^2-1]$, we have
    \begin{equation*}
        |B_i|\geq \min\{\tfrac{2n}{3},(1+\tfrac{1}{128m})^i|U|\}.
    \end{equation*}
    Let $i=m^2$, we have $|B_{m^2}|\geq \tfrac{2n}{3}$, and therefore, combining $(ii)-(iii)$, we have $|B_{m^2}\cap V_{m^2}|>\tfrac{V}{2}$.
    Hence, $|B_{G-W}^{m^2}(U,V)|>\tfrac{V}{2}$ with probability at least $1-\exp\left(-\Omega(\tfrac{\kappa|U|}{m^{6}})\right)$.
\end{proof}

\begin{proof}[Proof of Lemma~\ref{randomball}]
    A set $U'\subseteq V(G)$ is a \emph{well-expanding set} in $G$ if $|N_{G-W}(U')|\geq m^8|U'|$. First, we claim that Lemma \ref{randomball} is true for all well-expanding sets. Given a non-empty well-expanding set $U_1\subseteq V(G)$ and a vertex set $W$ with $|W|\leq \tfrac{|U_1|}{m^x}$, Lemma \ref{randomball1} applied to $U_1$ with $\kappa=m^8$ implies that the following fails with probability at most $\exp\left(-\Omega(|U_1|m^2)\right)$
    %not \blue{$\exp\{-\Omega(\tfrac{|U_1|}{m^{13}})\}$}.
    \begin{equation}\label{inball}
        |B^{m^2}_{G-W}(U_1,V)|>\tfrac{|V|}{2}.
    \end{equation}
    By union bound, we have
\begin{align}
\sum_{(U_1,W)}\exp\left(-\Omega(|U_1|m^2)\right)&\leq \sum_{j=k}^n\binom{n}{j}\binom{n}{j/m^x}\exp\left(-\Omega(jm^2)\right)
\nonumber \\
&\leq \sum_{j=k}^n (\tfrac{en}{j})^{2j}\exp\left(-\Omega(jm^2)\right) \nonumber \\
&\leq \sum_{j=k}^n \exp\left(3j\log\tfrac{n}{k}-\Omega(jm^2)\right)=\sum_{j=k}^n \exp\left(-\Omega(jm^2)\right)=o\left(\tfrac{1}{n}\right), \nonumber
\end{align}
Thus, with probability $1-o\left(\tfrac{1}{n}\right)$, we obtain that (\ref{inball}) holds for every well-expanding set $U_1$ and $W$ with $|W|\leq \tfrac{|U_1|}{m^x}$, as claimed.

Let $U_2\subseteq V(G)$ with $km^9\le |U_2|\leq \tfrac{2n}{3}$ and let $W\subseteq V(G)$ with $|W|\leq \tfrac{|U_2|}{m^x}$. By Lemma \ref{prop18} with $\kappa=m^{8}$, there is a set $U_2'\subseteq U_2$ such that $|N_{G}(U_2')| \geq m^8|U_2'|$ and $|U_2'|\geq \tfrac{|U_2|}{m^9}\ge k$.
Therefore as $|W|\leq \tfrac{|U_2|}{m^x}\le \tfrac{|U_2'|}{m^{x-9}}$ and $x\ge 11$, one can apply Lemma~\ref{randomball1} with $(x, U, W)=(x-9, U_2', W)$ to get
\[
    |B_{G-W}^{m^2}(U_2,V)|\geq |B_{G-W}^{m^2}(U_2',V)|>\tfrac{|V|}{2}.
\]
\end{proof}
\end{appendix}

\end{document}